\documentclass[final,twoside,12pt]{article}
\usepackage[in]{fullpage}

\usepackage[usenames]{color}
\usepackage{amsmath,amssymb}
\usepackage{amsthm}
\usepackage{bm}
\usepackage{algorithm}
\usepackage{algorithmic}

\usepackage{graphicx}

\newcommand{\bnu}{\ensuremath{\boldsymbol{\nu}}}
\newcommand{\bxi}{\ensuremath{\boldsymbol{\xi}}}

\newcommand{\bomega}{\ensuremath{\boldsymbol{\omega}}}

\newcommand{\bc}{\ensuremath{\bm{c}}}

\newcommand{\be}{\ensuremath{\bm{e}}}
\newcommand{\bof}{\ensuremath{\bm{f}}}
\newcommand{\bn}{\ensuremath{\bm{n}}}

\newcommand{\bx}{\ensuremath{\bm{x}}}
\newcommand{\bom}{\ensuremath{\bm{m}}}
\newcommand{\bw}{\ensuremath{\bm{w}}}

\newcommand{\by}{\ensuremath{\bm{y}}}
\newcommand{\bz}{\ensuremath{\bm{z}}}
\newcommand{\bB}{\ensuremath{\bm{B}}}

\newcommand{\bI}{\ensuremath{\bm{I}}}
\newcommand{\bO}{\ensuremath{\bm{0}}}
\newcommand{\bX}{\ensuremath{\bm{X}}}
\newcommand{\bM}{\ensuremath{\bm{M}}}

\theoremstyle{plain}
\newtheorem{theorem}{Theorem}[section]
\newtheorem{lemma}[theorem]{Lemma}
\newtheorem{proposition}[theorem]{Proposition}

\newtheorem{remark}[theorem]{Remark}

\begin{document}
\author{D.~P.~Bourne\footnotemark[2] \and S.~M.~Roper\footnotemark[2]}
\title{Centroidal power diagrams, Lloyd's algorithm and applications to optimal location problems}

\renewcommand{\thefootnote}{\fnsymbol{footnote}}
\footnotetext[2]{School of Mathematics and Statistics, University of Glasgow, University Gardens, Glasgow, G12 8QW}
\renewcommand{\thefootnote}{\fnsymbol{arabic}}

\maketitle

\begin{abstract}
In this paper we develop a numerical method for solving a class of optimization problems known as \emph{optimal location} or \emph{quantization} problems.
The target energy can be written either in terms of atomic measures and the Wasserstein distance or in terms of weighted points and power diagrams (generalized Voronoi diagrams). The latter formulation is more suitable for computation. We show that critical points of the energy are \emph{centroidal power diagrams}, which are generalizations of centroidal Voronoi tessellations, and that they can be approximated by a generalization of Lloyd's algorithm (Lloyd's algorithm is a common method for finding centroidal Voronoi tessellations). We prove that the algorithm is energy decreasing and prove a convergence theorem. Numerical experiments suggest that the algorithm converges linearly. We illustrate the algorithm in two and three dimensions using simple models of optimal location and crystallization.
 In particular, we test a conjecture about the optimality of the BCC lattice for a simplified model of block copolymers.

\end{abstract}
%

\section{Introduction}
\label{sec:intro}
In this paper we derive and analyze a numerical method for minimizing a class energies that arise in economics (optimal location problems), electrical engineering (quantization), and materials science (crystallization and pattern formation). Applications are discussed further in \S\ref{sec:App}. These energies can be formulated either in terms of atomic measures and the Wasserstein distance, equation \eqref{eqn:energy_fundamental}, or in terms of generalized Voronoi diagrams, equation \eqref{eqn:energy}. These formulations are equivalent, but \eqref{eqn:energy_fundamental} is more common in the applied analysis literature (e.g., \cite{BouchitteEtAl}, \cite{ButtazzoSantambrogio}) and \eqref{eqn:energy} is more common in the computational geometry and quantization literature (e.g, \cite{Du1999}, \cite{GershoGray}). Importantly for us, formulation \eqref{eqn:energy} is much more convenient for numerical work. We work with formulation \eqref{eqn:energy} throughout the paper after first deriving it from \eqref{eqn:energy_fundamental} in \S\ref{sec:1.1} and \S\ref{sec:1.2}. We start from \eqref{eqn:energy_fundamental} rather than directly from \eqref{eqn:energy} in order to highlight the connection between the different communities.

\subsection{Wasserstein formulation of the energy}
\label{sec:1.1}
Let $\Omega$ be a bounded subset of $\mathbb{R}^d$, $d \ge 2$, and $\rho: \Omega \to [0,\infty)$ be a given density on $\Omega$. Let $f:[0,\infty)\to \mathbb{R}$. We consider the following class of discrete energies, which are defined on sets of weighted points $\{\bx_i,m_i \}_{i=1}^N \in (\Omega \times (0,\infty))^N$, $\bx_i \ne \bx_j$ if $i \ne j$:
\begin{equation}
\label{eqn:energy_fundamental}
F\left(\{\bm{x}_i,m_i\}\right)=\sum_{i=1}^N f(m_i)+d^2\left(\rho,\sum_{i=1}^N m_i\delta_{x_i}\right).
\end{equation}
The second term is the square of the Wasserstein distance between the density $\rho$ and the atomic measure $\sum_{i=1}^N m_i\delta_{x_i}$. It is defined below in equation \eqref{eq:Wass}.
This energy models, e.g., the problem of optimally locating resources (such as recycling points, polling stations, or distribution centres) in a city or country $\Omega$ with population density $\rho$. The points $\bx_i$ are the locations of the resources and the weights $m_i$ represent their size.
The first term of the energy penalizes the cost of building or running the resources. The second term penalizes the total distance between the population and the resources.
In our case the Wasserstein distance $d(\cdot,\cdot)$ can be defined by
\begin{multline}
\label{eq:Wass}
d^2\left(\rho,\sum_{i=1}^Nm_i\delta_{x_i}\right) = \\
\min_{T:\Omega \to \{ \bx_i \}_{i=1}^N} \left\{ \sum_{i=1}^N \int_{T^{-1}(\bx_i)} |\bx - \bx_i|^2 \rho(\bx) \, d \bx : \int_{T^{-1}(\bx_i)} \rho \, d \bx = m_i \; \forall \; i \right\}.
\end{multline}
See, e.g., \cite{Villani}.
In two dimensions the minimization problem \eqref{eq:Wass} can be interpreted as the following optimal partitioning problem: The map $T$ partitions, e.g., a city $\Omega$ with population density $\rho$ into $N$ regions, $\{ T^{-1}(\bx_i) \}_{i=1}^N$. Region $T^{-1}(\bx_i)$ is assigned to the resource (e.g., polling station) located at point $\bx_i$ of size $m_i$. The optimal map $T$ does this in such a way to minimize the total distance squared between the population and the resources subject to the constraint that each resource can meet the demand of the population assigned to it.

The Wasserstein distance is well-defined provided that the weights $m_i$ are positive and satisfy the mass constraint
\begin{equation}
\label{eq:con}
\sum_i m_i = \int_{\Omega} \rho(\bx) \, d \bx.
\end{equation}
It can be shown that $d(\cdot,\cdot)$ is a metric on measures and that it metrizes weak convergence of measures, meaning that if $\rho_n$ converges to $\rho$, then $d(\rho,\rho_n) \to 0$. See, e.g., \cite[Ch.~7]{Villani}.
It is not necessary to be familiar with measure theory or the Wasserstein distance since we will soon reformulate the minimization problem $\min F$ as a more elementary computational geometry problem involving generalized Voronoi diagrams (power diagrams).

The given data for the problem are $\Omega$, $f$, $\rho$. We assume that $f$ is twice differentiable and
\begin{equation}
\label{eq:assump}
\Omega \textrm{ is convex}, \quad f'' \le 0, \quad f(0) \ge 0, \quad \rho \in C^0(\Omega), \quad \rho \ge 0.
\end{equation}
We also exclude linear functions $f(m)=a m$, $a \in \mathbb{R}$, since otherwise the first term of the energy is a constant,
 $\sum_i f(m_i) = a \sum_i m_i = \int_\Omega \rho \, d \bx$, and $F$ has no minimizer (see below). However, affine functions $f(m)=am + b$, $b > 0$, are admissible. The necessity and limitations of assumptions \eqref{eq:assump} are discussed in \S \ref{sec:Lim}.

The number $N$ of weighted points is \emph{not} prescribed and is an unknown of the problem: The goal is to minimize $F$ over sets of weighted points  $\{\bx_i,m_i \}_{i=1}^N$, subject to the constraint \eqref{eq:con}, and over $N$. The optimal value of $N$ is determined by the competition between the two terms of $F$. Amongst finite $N$, the first term is minimized when $N=1$, due to the concavity of $f$.
The infimum of the second term is zero, which is obtained in the limit $N \to \infty$ (this is because the measure $\rho \, d \bx$ can be approximated arbitrarily well with dirac masses, e.g., by using a convergent quadrature rule, and because the Wasserstein distance $d(\cdot,\cdot)$ metrizes weak convergence of measures).

Energies of the form of $F$ and generalizations have received a great deal of attention in the applied analysis literature, e.g., \cite{ButtazzoSantambrogio} and \cite{BouchitteEtAl} study the existence and properties of minimizers for broad classes of optimal location energies.
 There is far less work, however, on numerical methods for such problems. Exceptions include the case of \eqref{eqn:energy_fundamental} with $f=0$, which has been well-studied numerically. This is discussed in \S \ref{sec:CVTs}.

\subsection{Power diagram formulation of the energy}
\label{sec:1.2}
Minimizing $F$ numerically is challenging due to presence of the Wasserstein term, which is defined implicitly in terms of the solution to the optimal transportation problem \eqref{eq:Wass}. This is an infinite-dimensional linear programming problem in which every point in $\Omega$ has to be assigned to one of the $N$ weighted points $(\bx_i,m_i)$. Therefore even evaluating the energy $F$ is expensive. One option is to discretize $\rho$ so that \eqref{eq:Wass} becomes a finite-dimensional linear programming problem. This is still costly, however, and it turns out that by exploiting a deep connection between optimal transportation theory and computational geometry we can reformulate the minimization problem $\min F$ in such a way that we can avoid solving \eqref{eq:Wass} altogether.

First we need to introduce some terminology from computational geometry.
The \emph{power diagram} associated to a set of weighted points
$\{ \bx_i , w_i \}_{i=1}^N$, where $\bx_i \in \Omega$, $w_i \in \mathbb{R}$, is the collection of subsets $P_i \subseteq \Omega$ defined by
\begin{equation}
\label{eq:pd}
P_i=\{\bm{x}\in\Omega : \left|\bm{x}-\bm{x}_i\right|^2-w_i\le\left|\bm{x}-\bm{x}_k\right|^2-w_k \; \forall \; k\}.
\end{equation}
The individual sets $P_i$ are called \emph{power cells} (or cells) of the power diagram.
The power diagram is sometimes called the Laguerre diagram, or the radical Voronoi diagram.
If all the weights $w_i$ are equal we obtain the standard Voronoi diagram, see Figure \ref{fig:vorpow}. From equation \eqref{eq:pd} we see that the power cells $P_i$ are obtained by intersecting half planes and are therefore convex polytopes (or the intersection of convex polytopes with $\Omega$ in the case of cells that touch $\partial \Omega$): in dimension $d=3$ the cells are convex polyhedra, in dimension $d=2$ the cells are convex polygons. Note that some of the cells may be empty. The classical reference on generalized Voronoi diagrams is \cite{Okabe2000}.
\begin{figure}
\includegraphics[width=0.5\textwidth]{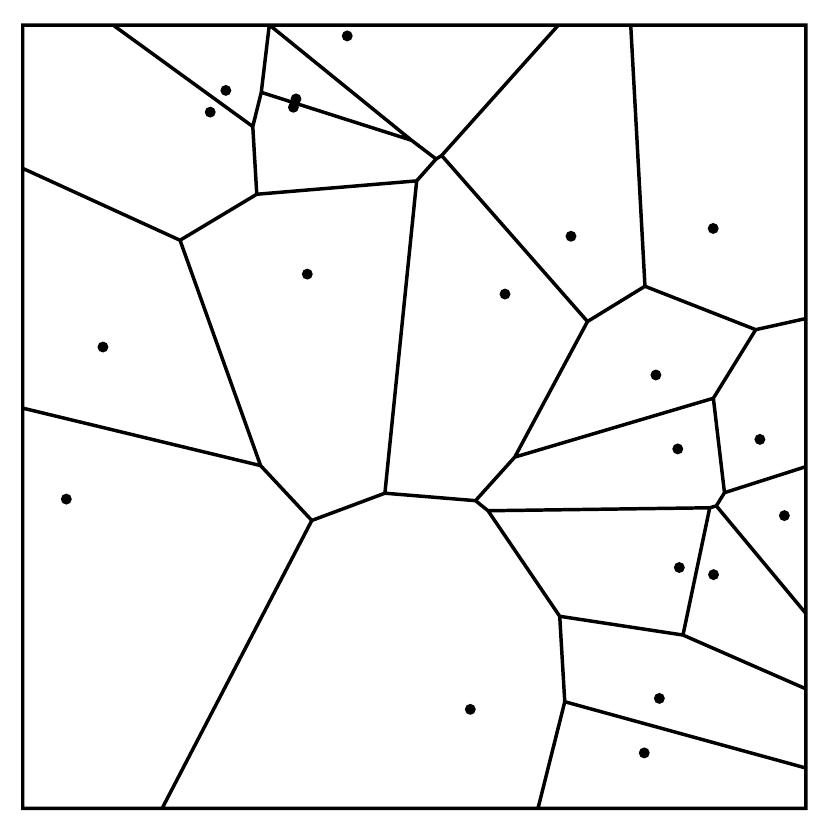}
\includegraphics[width=0.5\textwidth]{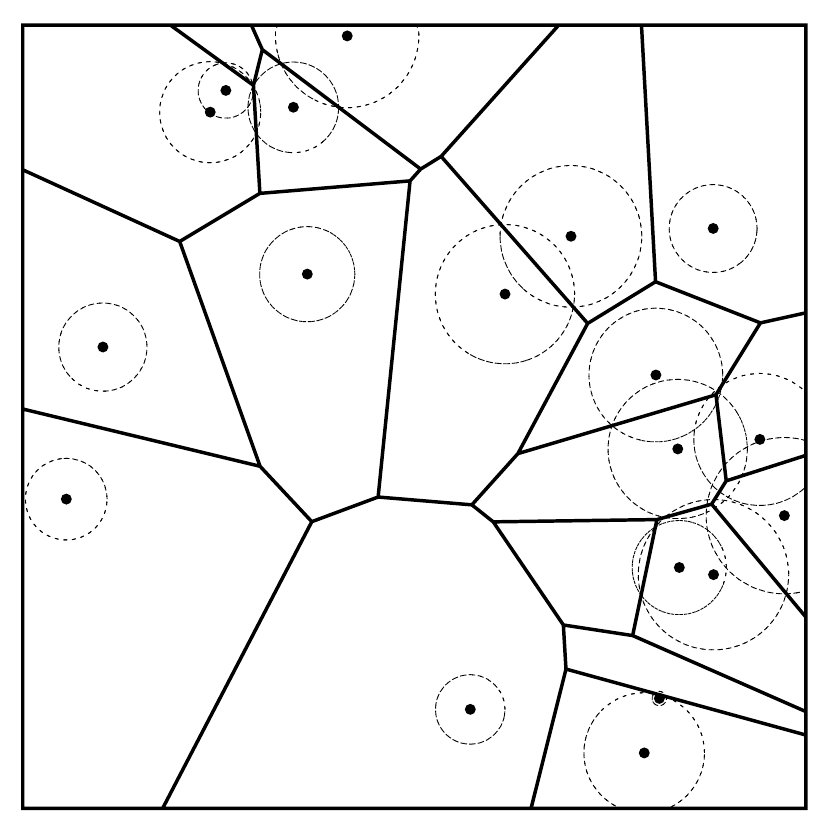}
\caption{\label{fig:vorpow}A comparison of a standard Voronoi diagram (left) with a power diagram (right). The location of the generators in both cases is the same, but the power diagram carries additional structure via the weights associated with each generator. The size of the weights in the power diagram is indicated by the radii of the dashed circles. Notice that in the power diagram it is possible for the generator to lie outside the cell or for the cell associated with a generator to be empty (the Voronoi diagram has 20 cells and the power diagram has 19 cells). The geometrical construction of the power diagram in terms of the generator locations and the circles is simple; for each point $\bm{x}$ construct a tangent line from $\bm{x}$ to the circles centred at $\bm{x}_i$ with radii $r_i$, the length of the tangent line is called the \emph{power} of the point $\bm{x}$, the point $\bm{x}$ belongs to the power cell that has minimum power. The weights of the generators in this case are $w_i=-r_i^2$.}
\end{figure}

Given weighted points $\{ \bx_i, m_i \}_{i=1}^N \in (\Omega \times (0,\infty))^N$, let $T_*$ be the minimizer in \eqref{eq:Wass}.
The optimal transport regions $\{ T_*^{-1}(\bx_i) \}_{i=1}^N$ form a power diagram:
There exits $\{w_i\}_{i=1}^N \in \mathbb{R}^N$
such that the power diagram $\{P_i \}_{i=1}^N$ generated by $\{ \bx_i, w_i \}_{i=1}^N$ satisfies $P_i = T_*^{-1}(\bx_i)$ for all $i$
(up to sets of $\rho \, d \bx$--measure zero). Conversely, if
$\{ P_i \}_{i=1}^N$ is any power diagram with generators $\{ \bx_i, w_i \}_{i=1}^N$, then
\begin{equation}
\label{eq:WassPower}
d^2\left(\rho,\sum_{i=1}^N m_i\delta_{x_i}\right) = \sum_{i=1}^N \int_{P_i} |\bx - \bx_i |^2 \rho \, d \bx \quad \textrm{where} \quad
m_i = \int_{P_i} \rho(\bx) \, d \bx.
\end{equation}
These results can be shown using Brenier's Theorem \cite[Thm.~2.12]{Villani} or the Kantorovich Duality Theorem \cite[Thm.~1.3]{Villani}.
See \cite[Thm.~1 \& 2]{Merigot} or \cite[Prop.~4.4]{BournePeletierRoper}. As far as we are aware these results first appeared in \cite{Aurenhammer98}, although not stated in the language of Wasserstein distances.

Equation \eqref{eq:WassPower} gives an explicit formula for the Wasserstein distance, without the need to solve a linear programming problem, provided that the weights $m_i$ can be written as $\int_{P_i} \rho(\bx) \, d \bx$ for some power diagram $\{ P_i \}$ (with generating points $\bx_i$). In practice actually finding this power diagram involves solving another linear programming problem
(the generating weights $w_i$ come from the solution to the dual linear programming problem to \eqref{eq:Wass}, see \cite[Prop.~4.4]{BournePeletierRoper}), but in our case this can be avoided since we are interested in minimizing $F$ rather than evaluating it at any given point.

We use this connection between the Wasserstein distance and power diagrams to rewrite the energy $F$ in new variables, changing variables from $\{ \bx_i, m_i \}_{i=1}^N \in (\Omega \times (0,\infty))^N$ to $\{ \bx_i, w_i \}_{i=1}^N \in (\Omega \times \mathbb{R})^N$.
By the results above, minimizing $F$ is equivalent to minimizing
\begin{equation}
\label{eqn:energy}
\boxed{ E \left( \{\bx_i,w_i\} \right) = \sum_{i=1}^N \left\{ f(m_i)+\int_{P_i} |\bx-\bx_i|^2 \rho(\bx) \,d \bx \right\} }
\end{equation}
where $\{P_i \}$ is the power diagram generated by $\{\bx_i,w_i\}$ and $m_i := \int_{P_i} \rho \, d \bx$.
The equivalence of $E$ and $F$ is in the following sense:
Given $\{ \bx_i, w_i \}_{i=1}^N \in (\Omega \times \mathbb{R})^N$ and the corresponding power diagram $\{ P_i \}_{i=1}^N$, equation \eqref{eq:WassPower} implies that
\[
E \left( \{\bx_i,w_i\} \right) = F \left( \left\{\bx_i, m_i \right\} \right)
\quad \textrm{for} \quad
m_i = \int_{P_i} \rho(\bx) \, d \bx.
\]
Conversely, it can be shown (e.g., \cite[Prop.~4.4]{BournePeletierRoper}) that given any $\{ \bx_i, m_i \}_{i=1}^N \in (\Omega \times (0,\infty))^N$, there exists
 $\{ w_i \}_{i=1}^N \in \mathbb{R}^N$
 such that the power diagram $\{ P_i \}_{i=1}^N$ generated by $\{ \bx_i, w_i \}_{i=1}^N$ satisfies $\int_{P_i} \rho(\bx) \, d \bx = m_i$ for all $i$.
 Then it follows from \eqref{eq:WassPower} that
$F \left( \{\bx_i,m_i\} \right) = E \left( \{\bx_i,w_i\} \right)$. The weights $\{ w_i \}_{i=1}^N \in \mathbb{R}^N$ are unique up to the addition
of a constant; it is easy to see from \eqref{eq:pd} that $\{ w_i + c \}_{i=1}^N$ and $\{ w_i \}_{i=1}^N$ generate the same power diagram.

While the energies $E$ and $F$ are equivalent, from a numerical point of view it is far more practical to work with $E$ since it can be easily evaluated, unlike $F$, since computing power diagrams is easy while solving the linear programming problem \eqref{eq:Wass} is not. In the rest of the paper we focus on finding local minimizers of $E$.
\subsection{Centroidal power diagrams and a generalized Lloyd algorithm}
\label{sec:genLloyd}
From now on we will write $(\bX,\bw)=((\bx_1, \ldots, \bx_N),(w_1, \ldots , w_N)) \in \Omega^N \times \mathbb{R}^N$ to denote the generators of a power diagram. In this section we introduce an algorithm for finding critical points of $E=E(\bX,\bw)$.

Let $\mathcal{G}^N \subset \Omega^N \times \mathbb{R}^N$ be the smaller class of generators such that no two generators coincide and there are no empty cells:
\begin{equation}
\label{eqn:G}
\mathcal{G}^N = \{ (\bX,\bw) \in \Omega^N \times \mathbb{R}^N : (\bx_i,w_i)\ne(\bx_j,w_j) \textrm{ if } i \ne j, \, P_i \ne \emptyset \;
\forall \; i \}.
\end{equation}
Define $\bxi : \mathcal{G}^N \to \Omega^N$ and $\bomega : \mathcal{G}^N \to \mathbb{R}^N$ by
\begin{equation}
\nonumber
\label{eqn:lloydmap}
\bxi(\bX,\bw) := (\bxi_1(\bX,\bw),\ldots,\bxi_N(\bX,\bw)), \quad
\bomega(\bX,\bw) := (\omega_1(\bX,\bw),\ldots,\omega_N(\bX,\bw)),
\end{equation}
where
\begin{equation}
\bxi_i(\bX,\bw)  := \frac{1}{m_i(\bX,\bw)} \int_{P_i(\bX,\bw)} \bx \rho (\bx) \, d \bx, \quad 
\omega_i(\bX,\bw)  := -f'(m_i(\bX,\bw)).
\end{equation}
Here $P_i(\bX,\bw)$ is the $i$-th power cell in the power diagram generated by $(\bX,\bw)$ and $m_i(\bX,\bw)$ is its mass:
\[
m_i(\bX,\bw) = \int_{P_i(\bX,\bw)} \rho(\bx) \, d \bx.
\]
Note that $\bxi_i(\bX,\bw)$ is the centroid (or centre of mass) of the $i$-th power cell. We will sometimes denote this by $\overline{\bx}_i$.
In \S \ref{sec:deriv} we show that critical points of $E$ are fixed points of the Lloyd maps:
\[
\nabla E(\bX,\bw) = \bO
\quad \iff \quad
(\bxi(\bX,\bw),\bomega(\bX,\bw)) = (\bX,\bw)
\]
(up to the addition of a constant vector to $\bw$ -- see Proposition \ref{prop:critE} for a precise statement).
The condition $\bxi(\bX,\bw)=\bX$ means that the power diagram generated by $(\bX,\bw)$ has the property that
$\bx_i$ is the centroid of its power cell $P_i$ for all $i$. We call these special types of power diagrams \emph{centroidal power diagrams}. This is in analogy with centroidal Voronoi tessellations (CVTs), which are special types of Voronoi diagrams with the property that the generators of the Voronoi diagram are the centroids of the Voronoi cells. See \cite{Du1999} for a nice survey of CVTs. Note also that CVTs can be viewed as a special type of centroidal power diagram where all the weights are equal, $w_i=c$ for all $i$, $c \in \mathbb{R}$, since power diagrams with equal weights are just Voronoi diagrams.

The following algorithm is an iterative method for finding fixed points of $(\bxi,\bomega)$, and therefore critical points of $E$:
\begin{algorithm}[H]
\label{algo_lloyd}
\textbf{Initialization:} Choose $N_0 \in \mathbb{N}$ and $(\bX^0,\bw^0) \in \mathcal{G}^{N_0}$. \newline
\textbf{At each iteration:}
\begin{itemize}
\item[\textbf{(1)}] Update the generators: Given $(\bX^k,\bw^k) \in \mathcal{G}^{N_k}$, compute the corresponding power diagram
and define $(\bX^{k+1},\bw^{k+1}) \in \Omega^{N_k} \times \mathbb{R}^{N_k}$ by
\[
\bX^{k+1} = \bxi(\bX^k,\bw^k), \quad \bw^{k+1} = \bomega(\bX^k,\bw^k).
\]
\item[\textbf{(2)}] Remove empty cells: Compute the power diagram $\{P_i^{k+1}\}_{i=1}^{N_k}$ generated by $(\bX^{k+1},\bw^{k+1})$ and let
 \[
 J = \left\{ j \in \{1,\ldots,N_k\} : P^{k+1}_j = \emptyset \right\}.
 \]
 For all $j \in J$, remove $(\bx_j^{k+1},w_j^{k+1})$ from the list of generators. Then replace $N_k$ with $N_{k+1}=N_k - |J|$.
\end{itemize}
\caption{\label{algo:genLloyd}The generalized Lloyd algorithm for finding critical points of $E$}
\end{algorithm}
In particular this algorithm computes centroidal power diagrams, and it is a generalization of Lloyd's algorithm \cite{Lloyd}, which is a popular method for computing centroidal Voronoi tessellations. See \cite{Du1999}.
The classical Lloyd algorithm is recovered from our generalized Lloyd algorithm by simply taking the weights to be constant at each iteration, e.g., $\bw^k = \bO$ for all $k$. Due to this relation, we refer to $\bxi$ and $\bomega$ as generalized Lloyd maps.

Step (2) of the algorithm means that, given $N_0 \in \mathbb{N}$ and $(\bX^0,\bw^0) \in \mathcal{G}^{N_0}$, the algorithm can converge to a fixed point $(\bX,\bw) \in \mathcal{G}^{N}$ with $N < N_0$. This means that the algorithm can partly correct for an incorrect initial guess $N_0$ (recall that we are minimizing $E(\bX,\bw)$ over $(\bX,\bw) \in \mathcal{G}^N$ and over $N$). It is still possible, however, that the algorithm converges to a local minimizer of $E$, possibly with a non-optimal value of $N$. Note also that the algorithm can eliminate generators, but it cannot create them. Therefore it is impossible for the algorithm to find a global minimizer of $E$ if the initial value of $N_0$ is less than the optimal value. We discuss
strategies for finding global as opposed to local minimizers in \S\ref{sec:implement} and \S\ref{sec:illus}.

Algorithm \ref{algo:genLloyd} was introduced for the special case of $d=2$, $\rho=1$, $f(m) = \sqrt{m}$ in \cite[Sec.~4]{BournePeletierRoper}.
In the current paper we extend it to the broader class of energies \eqref{eqn:energy}, analyze it (prove that it is energy decreasing and that it converges, Theorems \ref{thm:energyDecrease}, \ref{thm:conv}), and implement it in both two and three dimensions. In addition, the derivation here, unlike in \cite{BournePeletierRoper}, is accessible to those not familiar with measure theory and optimal transport theory
since we work with formulation \eqref{eqn:energy} rather than \eqref{eqn:energy_fundamental}.

\subsection{The case $f=0$ and $N$ fixed: CVTs and Lloyd's algorithm}
\label{sec:CVTs}
Setting $f=0$ in \eqref{eqn:energy_fundamental} and fixing $N$ gives the energy
\[
F_N\left(\{\bm{x}_i,m_i\}\right)=d^2\left(\rho,\sum_{i=1}^N m_i\delta_{x_i}\right).
\]
It is necessary to fix $N$ since otherwise this has no minimizer; the infimum is zero, which is obtained in the limit $N \to \infty$ by approximating
$\rho$ with dirac masses. It can be shown that minimizing $F_N$ is equivalent to minimizing
\[
E_N(\{\bx_i\}) = \sum_{i=1}^N \int_{V_i} |\bx - \bx_i |^2 \rho(\bx) \, d \bx
\]
where $\{ V_i \}_{i=1}^N$ is the Voronoi diagram generated by $\{ \bx_i \}_{i=1}^N$:
\[
V_i= \{ \bx \in \Omega : | \bx-\bx_i | \le  | \bx-\bx_k | \; \forall \; k \}.
\]
See \cite[Sec.~4.1]{BournePeletierRoper}. Numerical minimization of $E_N$ has been well-studied. A necessary condition for minimality is that
$\{ \bx_i \}_{i=1}^N$ generates a centroidal Voronoi tessellation (CVT). CVTs can be easily computed using the classical Lloyd algorithm. See, e.g., \cite{Du1999}. Convergence of the algorithm is studied in \cite{Du2006}, \cite{Du1999} and \cite{SabinGray}, among others, and there is a large literature on CVTs and Lloyd's algorithm. However, we are not aware of any work (other than \cite{BournePeletierRoper}) on numerical minimization of $E$ for $f \ne 0$.
\subsection{Applications}
\label{sec:App}
Energies of the form \eqref{eqn:energy}, or equivalently \eqref{eqn:energy_fundamental}, arise in many applications.

\subsubsection{Simple model of pattern formation: block copolymers}
\label{Subsubsec: block copolymer}
The authors first came in contact with energies of the form \eqref{eqn:energy_fundamental} in a pattern formation problem in materials science
\cite{BournePeletierRoper}. The following energy is a simplified model of phase separation for two-phase materials called block copolymers, for the case where one phase has a much smaller volume fraction than the other:
\begin{equation}
\label{eqn:block}
E \left( \{\bx_i,w_i\} \right) = \sum_{i=1}^N \left\{ \lambda  m_i^{\frac{d-1}{d}} + \int_{P_i} |\bx-\bx_i|^2 \,d \bx \right\}
\end{equation}
where $m_i = \int_{P_i} 1 \, dx = | P_i |$ and $d=2$ or $3$.
The measure $\nu = \sum_i m_i \delta_{\bx_i}$ represents the minority phase. In three dimensions, $d=3$, this represents $N$ small spheres of the minority phase centred at $\{ \bx_i \}_{i=1}^N$. The weights $m_i$ give the relative size of the spheres. These spheres are surrounded by a `sea' of the majority phase. In two dimensions, $d=2$, the measure $\nu$ represents $N$ parallel cylinders of the minority phases and $\Omega$ is a cross-section perpendicular to the axes of the cylinders.
The first term of $E$ penalizes the surface area between the two phases and so prefers phase separation ($N=1$), and the second term prefers phase mixing ($N=\infty$). The parameter $\lambda$ represents the repulsion strength between the two phases. Equation \eqref{eqn:block} is the special case of \eqref{eqn:energy} with $\rho=1$ and $f(m)=\lambda m^{\frac{d-1}{d}}$.

This energy can be viewed as a toy model of the popular Ohta-Kawasaki model of block copolymers (see, e.g., \cite{ChoksiPeletierWilliams}). Like the Ohta-Kawasaki energy,
it is non-convex and non-local (in the sense that evaluating $E$ involves solving an auxiliary infinite-dimensional problem).
Unlike the Ohta-Kawasaki energy, however, it is discrete, which makes it much more amenable to numerics and analysis. In general it can be viewed as a simplified model of non-convex, non-local energy-driven pattern formation, and it has applications in materials science outside block copolymers, e.g., to crystallization. It is also connected to the Ginzburg-Landau model of superconductivity \cite[p.~123--124]{BournePeletierTheil}.

In \cite{BournePeletierRoper} it was demonstrated numerically that for $d=2$ minimizers of $E$ tend to a hexagonal tiling as $\lambda \to 0$
(in the sense that the power diagram generated by $\{ \bx_i,w_i \}$ tends to a hexagonal tiling). This was proved in \cite{BournePeletierTheil}, and it agrees with block copolymer experiments, where in some parameter regime the minority phase forms hexagonally packed cylinders. It was conjectured in \cite{BournePeletierRoper} that for the case $d=3$, minimizers of $E$ tend to a body-centred cubic (BCC) lattice as $\lambda \to 0$ (meaning that
$\{ \bx_i \}$ tend to a BCC lattice and $w_i \to 0$). We examine this conjecture in \S \ref{Subsec:3D}. In particular, numerical minimization of $E$ in three dimensions suggests that the BCC lattice is at least a local minimizer of $E$ when $\Omega$ is a periodic box.
Again, this agrees with block copolymer experiments, where in some parameter regime the minority phase forms a BCC lattice.

\subsubsection{Quantization}
Energies of the form \eqref{eqn:energy} can be used for data compression using a technique called \emph{vector quantization}.
By taking $f=0$ in \eqref{eqn:energy} and evaluating the resulting energy at $w_i=0$ for all $i$, so that the power diagram $\{ P_i \}_{i=1}^N$ generated
by $\{\bx_i,0 \}_{i=1}^N$ is just the Voronoi diagram $\{V_i \}_{i=1}^N$ generated by $\{ \bx_i \}_{i=1}^N$, we obtain the energy
\begin{equation}
\label{eqn:D}
D(\{\bx_i \}) = \sum_{i=1}^N \int_{V_i} |\bx-\bx_i|^2 \rho(\bx) \,d \bx \equiv \int_\Omega \min_i | \bx - \bx_i |^2 \rho(\bx) \, d \bx.
\end{equation}
This is known in the quantization literature as the \emph{distortion}. See \cite[Sec.~33]{Gruber07} for a mathematical introduction to vector quantization and \cite{GershoGray} and \cite{GrayNeuhoff} for comprehensive treatments.
Roughly speaking, the points $\bx$ of $\Omega$ represent signals (e.g., parts of an image or speech) and $\bx_i$ represent codewords in the codebook $\{ \bx_i \}_{i=1}^N$.
The function $\rho$ is a probability density on the set of signals $\Omega$.
If a signal $\bx$ belongs to the Voronoi cell $V_i$, then the encoder assigns to it the codeword $\bx_i$, which is then stored or transmitted.
$D$ measures the quality of the encoder, the average distortion of signals. The minimum value of $D$ is called the \emph{minimum distortion}.

In practice distortion is minimized subject to a constraint on the number of bits in the codebook. The codewords $\bx_i$ are mapped to binary vectors before storage or transmission. In \emph{fixed-rate} quantization all these vectors have the same length. In \emph{variable-rate quantization} the length depends on the probability density $\rho$: Let $m_i = \int_{V_i} \rho \, d \bx$ be the probability that a signal lies in Voronoi cell $V_i$. If $m_i$ is large, then $\bx_i$ should be mapped to a short binary vector since it occurs often. For cells with lower probabilities, longer binary vectors can be used. The \emph{rate} of an encoder has the form
\[
R = \sum_{i=1}^N l_i m_i
\]
where $l_i$ is the length of the binary vector representing $\bx_i$. Note that $R$ is the expected value of the length.
Distortion $D$ is decreased by choosing more codewords. On the other hand, this means that the rate $R$, and hence the storage/transmission cost, is increased.
Optimal encoders can be designed by trading off distortion against rate by minimizing energies of the form
\[
\lambda R + D
\]
where $\lambda$ is a parameter determining the tradeoff. See \cite[p.~2342]{GrayNeuhoff}. Our energy \eqref{eqn:energy} generalises this: Take $l_i = l(1/m_i)$ for some concave function $l$ so that $m \mapsto l(1/m)m$ is concave.
In addition, $l$ should be increasing so that the code length decreases as the probability $m$ increases.
 We replace the Voronoi cells in \eqref{eqn:D} with power cells, which means that signals in power cell $P_i$ are mapped to codeword $\bx_i$. Then the energy $\lambda R + D$ has the form of \eqref{eqn:energy}:
\[
E(\{ \bx_i,w_i \}) = \sum_{i=1}^N \left\{ f(m_i)+\int_{P_i} |\bx-\bx_i|^2 \rho(\bx) \,d \bx \right\} \quad \textrm{where} \quad f(m) = \lambda l\left( \dfrac 1m \right) m.
\]


\subsubsection{Optimal location of resources}
As discussed in \S\ref{sec:1.1} and \S\ref{Subsec: Non-const}, energies of the form \eqref{eqn:energy_fundamental} and \eqref{eqn:energy} can be used to model the optimal location of resources $\{ \bx_i \}$ in a city or country $\Omega$ with population density $\rho$. The resources have size $m_i$, serve region $P_i$, and cost $f(m_i)$ to build or run. The assumption that $f$ is concave (introduced for mathematical convenience to prove Theorem \ref{thm:energyDecrease}) is also natural from the modelling point of view since it corresponds to an economy of scale. The energy trades off building/running costs against distance between the population and the resources.

\subsubsection{Other applications and connections}
  Energies of the form \eqref{eqn:energy}, usually with $f=0$, also arise in data clustering and pattern recognition ($k$-means clustering) \cite{Hartigan}, \cite{MacQueen}, image compression (this is a special case of vector quantization) \cite[Sec.~2.1]{Du1999}, numerical integration \cite[Sec.~2.2]{Du1999}, \cite[p.~497--499]{Gruber07} and convex geometry (packing and covering problems, approximation of convex bodies by convex polytopes) \cite[Sec.~33]{Gruber07}. Taking $f \ne 0$ in \eqref{eqn:energy} gives the algorithm more freedom, e.g., to automatically select the number of data clusters in addition to their location, based on a cost per cluster.

  Voronoi diagrams have recently gained a lot of interest in the materials science community, e.g., to model solid foams \cite{Harrison} and grains in metals \cite{Kok}, although this is usually done in a more heuristic manner than by energy minimization. Global minimizers of $E$ can be difficult to find if they have a large value of $N$, and the generalized Lloyd algorithm tends to converge to local minimizers. These often resemble grains in metals, see Figure \ref{fig:flatness_large}, which suggests that energy minimization might be a good method to produce Representative Volume Elements for the finite element simulation of materials with microstructure.

  Several important PDEs, such as the heat equation and Fokker-Plank equation, can be written as a time-discrete gradient flow of an energy with respect to the Wasserstein distance \cite{JordanKinderlehrerOtto}. For example, for the heat equation, an energy related to \eqref{eqn:energy_fundamental} is minimized at every time step, with the important differences that the first term of the energy is the integral of a convex function (as opposed to the sum of a concave function) and the second term is the Wasserstein distance between two absolutely continuous measures (as opposed to between an absolutely continuous measure and an atomic measure). A spatial discretization would bring the second terms in line and replace the integral in the first term by a sum. It could be argued, however, that we do not need another numerical method to solve the linear heat equation.  Energies involving the Wasserstein distance also arise in models of dislocation dynamics \cite{Lucia}.

\subsection{Limitations of the algorithm}
\label{sec:Lim}
First we discuss the assumptions on the data given in equation \eqref{eq:assump}.

The assumption that $\Omega$ is convex ensures that
the centroid of each power cell lies in $\Omega$. Without this assumption the algorithm could produce an unfeasible solution with $\bx_i \notin \Omega$
for some $i$. For example, if $\Omega$ is the annulus $A(r_1,r_2)$ centred at the origin, $\rho=1$, and $f$ is chosen suitably, then $E$ is minimized when $N=1$ by $(\bx_1,w_1)$ in which $|\bx_1|=r_1$ (the generator lies on the interior boundary of the annulus) and $w_1$ is irrelevant (in the case where there is only one cell the weight is not determined).
The generalized Lloyd algorithm, however, initialised with $N_0=1$, would return $\bx = \bO \notin \Omega$. This strong limitation on the shape of $\Omega$ means that the algorithm cannot be used to solve optimal location problems in highly nonconvex countries like Scotland. We plan to address this issue in a future paper.

The concavity assumption on $f$, $f'' \le 0$, is necessary to prove Theorem \ref{thm:energyDecrease}, which asserts that step (1) of the algorithm decreases the energy at every iteration. As discussed in \S \ref{sec:App}, it is also a reasonable modelling assumption for many applications.
The assumption that $f(0)\ge 0$ ensures that iteration step (2) is also energy decreasing.

If $f$ is convex then the energy behaves very differently and the generalized Lloyd algorithm may not be suitable. The first term is not necessarily minimized when $N=1$, but when all the power cells have the same mass, since by Jensen's inequality
\[
\sum_{i=1}^N f(m_i) \ge N f \left( \dfrac{M}{N} \right) \quad \textrm{where} \quad M = \int_{\Omega} \rho \, d \bx.
\]
If in addition $f \ge 0$ and $N f(M/N) \to 0$ as $N \to \infty$, then $E$ does not have a global minimizer.
Its infimum is zero, obtained in the limit $N \to \infty$
by approximating $\rho$ arbitrarily well by dirac masses. We have not studied the case where $f$ is neither concave nor convex.

As discussed in \S \ref{sec:genLloyd}, another limitation of the algorithm is that, while it can annihilate generators, step (2), it cannot create them. Therefore the initial guess $N_0$ for the optimal number of generators should be an over estimate. This limitation could be addressed by using a simulated annealing method to randomly introduce new generators at certain iterations. This could also be used to prevent the algorithm from getting stuck at a local minimizer.

\subsection{Generalizations}
While we have focussed on energy \eqref{eqn:energy}, our general methodology could be easily applied to broader classes of optimal location energies where the first term is more general, e.g., to
\[
E \left( \{\bx_i,w_i\} \right) = g(\{\bx_i,m_i\})+ \sum_{i=1}^N \int_{P_i} |\bx-\bx_i|^2 \rho(\bx) \,d \bx
\]
where $m_i = \int_{P_i} \rho \, d \bx$.

Our algorithm can also be modified to minimize the following energy, which is obtained from \eqref{eqn:energy_fundamental} by replacing the square of the 2-Wasserstein distance with the $p$-th power of the $p$-Wasserstein distance, $p \in [1,\infty)$:
\[
F_p\left(\{\bm{x}_i,m_i\}\right)=\sum_{i=1}^N f(m_i)+d^p_p\left(\rho,\sum_{i=1}^N m_i\delta_{x_i}\right).
\]
See \cite[Chap.~7]{Villani} for the definition of $d_p(\cdot,\cdot)$.
In this case the energy can be rewritten in terms of what we call \emph{$p$-power diagrams}.
These are a generalization of power diagrams where the cells generated by $\{ \bx_i, w_i \}$ are defined by
\[
P_i=\{\bm{x}\in\Omega : \left|\bm{x}-\bm{x}_i\right|^p-w_i\le\left|\bm{x}-\bm{x}_k\right|^p-w_k\;\forall \; k\}.
\]
For $p=2$ this is just the power diagram.
For $p=1$ this is known as the Appollonius diagram (or the additively weighted Voronoi diagram, or the Voronoi diagram of disks).
For general $p$ there does not seem to be a standard name, although they fall into the class of
generalized Dirichlet tessellations, or generalized additively weighted Voronoi diagrams.
It can be shown that minimizing $F_p$ is equivalent to minimizing
\begin{equation}
E_p \left( \{\bx_i,w_i\} \right) = \sum_{i=1}^N \left\{ f(m_i)+\int_{P_i} |\bx-\bx_i|^p \rho(\bx) \,d \bx \right\}
\end{equation}
where $\{P_i \}$ is the $p$-power diagram generated by $\{\bx_i,w_i\}$ and $m_i := \int_{P_i} \rho \, d \bx$.
 See \cite[Sec.~4.2]{BournePeletierRoper}. Critical points of $E_p$ can be found using a modification of the generalized Lloyd algorithm where
 for each $i$ the map $\bxi_i$ returns the $p$-centroid of the $p$-power cell $P_i$, i.e., $\bxi_i(\bX,\bw)$ satisfies the equation
 \begin{equation}
 \label{eq:p_cent}
 \int_{P_i} (\bxi_i-\bx)|\bxi_i-\bx|^{p-2} \, dx = \bO.
 \end{equation}
 See \cite[Th.~4.16]{BournePeletierRoper}.
 For the case $p=2$ this equation just says that $\bxi_i$ is the centroid of $P_i$.
 Therefore in principle the algorithm can be extended to all $p \in [1,\infty)$. In practice it is much harder to implement.
 Except for the cases $p=1,2$, we are not aware of any efficient algorithms for computing $p$-power diagrams. This is due to the fact that
 for $p \ne 2$ the boundaries between cells are curved  (unless all the weights are equal).
 In addition, evaluating the Lloyd map $\bxi(\bX,\bw)$ involves solving the nonlinear equation \eqref{eq:p_cent}. We plan to say more about this aspects in a future paper.

\subsection{Structure of the paper}
The generalized Lloyd algorithm, Algorithm \ref{algo:genLloyd}, is derived in \S\ref{sec:deriv}. In \S\ref{sec:prop} we prove that it is energy decreasing, prove a convergence theorem, and study its structure. Implementation issues, such as how to compute power diagrams, are discussed in \S\ref{sec:implement}. Numerical illustrations in two and three dimensions are given in \S\ref{sec:illus}. In the appendix we give some useful formulas for implementing the algorithm for the special case $\rho=$ constant, in which case it is not necessary to use a quadrature rule.


\section{Derivation of the algorithm}
\label{sec:deriv}

In this section we derive the generalized Lloyd algorithm, Algorithm \ref{algo:genLloyd}, which is a fixed point method for the calculation of stationary points of the energy $E$, defined in equation \eqref{eqn:energy}. Calculating the gradient of $E$ requires care since this involves differentiating the integrals appearing in the definition of $E$ with respect to their domains. We perform this calculation in \S\ref{Subsec: H} and \S\ref{Subsec: Crit points}, after introducing some notation in \S\ref{Subsec: Notation}.

\subsection{Notation for power diagrams}
\label{Subsec: Notation}
Throughout this paper we take $\Omega$ to be a bounded, convex subset of $\mathbb{R}^d$, $d \ge 2$.
We will take $d=2$ or $3$ for purposes of illustration, but the theory developed applies for all $d\ge 2$.

Given weighted points
$(\bX,\bw)=((\bx_1, \ldots, \bx_N),(w_1, \ldots , w_N))
\in \Omega^N \times \mathbb{R}^N$ and the associated power diagram $\{P_i\}_{i=1}^N$ (defined in equation \eqref{eq:pd}), we introduce the following notation:
\begin{gather}
\label{eqn:def1}
\quad d_{ij} = | \bx_j - \bx_i |, \qquad \bn_{ij} = \frac{\bx_j - \bx_i}{d_{ij}}, \qquad F_{ij} = P_i \cap P_j, \\
\label{eqn:def2}
m_i = \int_{P_i} \rho(\bx) \, d \bx, \qquad m_{ij} = \int_{F_{ij}} \rho(\bx) \, d \bx, \\
\label{eqn:def3}
\overline{\bx}_i = \frac{1}{m_i} \int_{P_i} \bx \rho (\bx) \, d \bx, \qquad \overline{\bx}_{ij} = \frac{1}{m_{ij}} \int_{F_{ij}} \bx \rho (\bx) \, d \bx, \\
\label{eqn:def4}
J_i = \{ j \ne i : P_i \cap P_j \ne \emptyset \}.
\end{gather}
Here $d_{ij}$ is the distance between points $\bx_i$ and $\bx_j$; $\bn_{ij}$ is the unit vector pointing from $\bx_i$ to $\bx_j$; the set $F_{ij}$ is the \emph{face} common to both cells $P_i$ and $P_j$; $m_i$ is the mass of cell $P_i$; $m_{ij}$ is the mass of face $F_{ij}$; $\overline{\bx}_i$ is the \emph{centre of mass} of the cell $P_i$ and $\overline{\bx}_{ij}$ is the centre of mass of face $F_{ij}$. The set of indices of the neighbours of cell $P_i$ is given by the index set $J_i$.
In the case $d=2$ the power cells are convex polygons and rather than referring to the intersections of neighbouring cells as faces, we refer to them as edges.

Recall that we sometimes write $P_i(\bX,\bw)$ for the power cells generated by $(\bX,\bw)$,
instead of simply $P_i$, to emphasize that the power diagram is generated by $(\bX,\bw)$. Similarly, we will sometimes write $m_i(\bX,\bw)$ for the mass of the $i$-th power cell. From equation \eqref{eq:pd} it is easy to see that adding a constant $c \in \mathbb{R}$ to all the weights generates the same power diagram: $P_i(\bX,\bw+\bc)=P_i(\bX,\bw)$ for all $i$, where $\bc = (c, \ldots , c ) \in \mathbb{R}^N$.
Let $\mathbb{R}_+ = [0,\infty)$ and let $\bom : \Omega^N \times \mathbb{R}^N \to \mathbb{R}_{+}^N$ be the function defined by
\begin{equation}
\label{eqn:m}
\bom (\bX,\bw) = (m_1(\bX,\bw), \ldots, m_N(\bX,\bw)),
\end{equation}
which gives the mass of all of the cells generated by $(\bX,\bw)$. Note that some of the cells may be empty (at most $N-1$ of them), in which case the corresponding components of $\bom$ take the value zero. Given a density $\rho : \Omega \to [0,\infty)$, let the space of admissible masses be
\begin{equation}
\mathcal{M}^N = \left\{ \bM \in \mathbb{R}_+^N : \sum_{i=1}^N M_i = \int_\Omega \rho(\bx) \, d \bx \right\}.
\end{equation}

Throughout this paper $\bI_m$ denotes the $m$-by-$m$ identity matrix.

\subsection{The helper function $H$}
\label{Subsec: H}

Motivated by \cite{Du2006}, where convergence of the classical Lloyd algorithm is studied, we introduce a helper function
$H$ 
defined by
\begin{multline}
H\left((\bX^1,\bw^1),(\bX^2,\bw^2),\bM \right) := \\
\sum_{i=1}^N \left\{ M_i w^1_i + f(M_i)
+ \int_{P_i(\bX^2,\bw^2)} ( |\bx-\bx^1_i|^2-w^1_i ) \rho (\bx) \, d \bx
\right\}
\end{multline}
where $(\bX^k,\bw^k) = ((\bx_1^k,\ldots,\bx_N^k),(w_1^k,\ldots,w_N^k))$ for $k \in \{1, 2 \}$, $\bM = (M_1,\ldots,M_N)$, and the domain of $H$ is
$(\Omega^N \times \mathbb{R}^N) \times (\Omega^N \times \mathbb{R}^N) \times \mathcal{M}^N$.
The energy $E$ is recovered by choosing the arguments of $H$ appropriately:
\begin{equation}
\label{eqn:E=H}
E\left( \bX,\bw \right) = H\left((\bX,\bw),(\bX,\bw),\bom(\bX,\bw)\right).
\end{equation}
Note that $H$ is invariant under addition of a constant to all the weights:
\begin{equation}
H\left((\bX^1,\bw^1+\bc_1),(\bX^2,\bw^2+\bc_2),\bM \right) = H\left((\bX^1,\bw^1),(\bX^2,\bw^2),\bM \right)
\end{equation}
for all $\bc_i = c_i (1, \ldots ,1) \in \mathbb{R}^N$, $i \in \{ 1,2 \}$, since $\bM \in \mathcal{M}^N$.

\begin{lemma}[Properties of $H$]
\label{lemma:H}
Let $\bxi$, $\omega$ be the Lloyd maps defined in equation \eqref{eqn:lloydmap}. Then
\begin{align}
\nonumber
& (i) \quad \min_{\bX^1 \in \Omega^N} H \left((\bX^1,\bw^1),(\bX^2,\bw^2), \bM \right) =
H \left( \left( \bxi (\bX^2,\bw^2),\bw^1 \right) , (\bX^2,\bw^2), \bM \right), \\
\nonumber
& (ii) \quad H \left( (\bX,\bw^1),(\bX,\bw),\bom(\bX,\bw) \right) = E\left( \bX,\bw \right),
\textrm{ i.e., is independent of } \bw^1,
\\
\nonumber
& (iii) \; \, \, H \left( (\bX^1,\bw^1), (\bX^2,\bw^2) ,\bM \right) \ge H \left((\bX^1,\bw^1),(\bX^1,\bw^1),\bM \right),
\\
\nonumber
&  \phantom{(iii)}
\textrm{ with equality if and only if } P_i(\bX^1,\bw^1)=P_i(\bX^2,\bw^2) \textrm{ for all } i,
\\
\nonumber
& (iv) \; \, \max_{\bM \in \mathbb{R}_+^N} H \left( \left( \bX^1, \bomega(\bX^2,\bw^2) \right),(\bX^2,\bw^2), \bM \right) =
\\
\nonumber
& \phantom{aaaaaaaaaaaaaaaaaaaaaaaaaaaaa} H \left( \left( \bX^1, \bomega(\bX^2,\bw^2) \right), (\bX^2,\bw^2), \bom(\bX^2,\bw^2) \right).
\end{align}
\end{lemma}

\begin{proof}
Property (i): For fixed $\bX^2 \in \Omega^N$, $\bw^1, \bw^2 \in \mathbb{R}^N$ and $\bM \in \mathcal{M}^N$, define the function
 $h:\Omega^N \to \mathbb{R}$ by $h(\bX^1) := H \left((\bX^1,\bw^1),(\bX^2,\bw^2), \bM \right)$.
Then
\[
\frac{\partial h}{\partial \bx_i^1} (\bX^1) = 2 \int_{P_i(\bX^2,\bw^2)} (\bx^1_i-\bx) \rho(\bx) \, d \bx = 2 m_i(\bX^2,\bw^2) (\bx_i^1-\bxi_i(\bX^2,\bw^2))
\]
by the definition \eqref{eqn:lloydmap} of $\bxi_i$.
Therefore $\bxi ( \bX^2,\bw^2 )$ is a critical point of $h$. Moreover it is a global minimum point since $h$ is convex:
\[
\frac{\partial^2 h}{\partial \bx_i^1 \partial \bx^1_j} = \left\{
\begin{array}{cl}
2 m_i(\bX^2,\bw^2) \bI_d & \textrm{if } i=j, \\
\bO & \textrm{if } i \ne j,
\end{array}
\right.
\]
where $\bI_d$ and $\bO$ are the $d$-by-$d$ identity and zero matrices.
(Note that $h$ is not necessarily strictly convex since $m_i(\bX^2,\bw^2)$ may be zero for some $i$, which is the case when the power cell
 $P_i(\bX^2,\bw^2)$ is empty.)

Property (ii) is immediate from the definitions of $H$ and $E$.

Property (iii): This follows from the fact that for any partition $\{ S_i \}_{i=1}^N$ of $\Omega$ we have
\[
\sum_i \int_{S_i} ( |\bx-\bx^1_i|^2-w^1_i ) \rho (\bx) \, d \bx
\ge
\sum_i \int_{P_i(\bX^1,\bw^1)} ( |\bx-\bx^1_i|^2-w^1_i ) \rho (\bx) \, d \bx
\]
with equality if and only if $\{ S_i \}_{i=1}^N$ is the power diagram generated by $(\bX^1,\bw^1)$ (up to sets of $\rho \, d \bx$--measure zero).
This follows since
\[
\sum_i \int_{P_i(\bX^1,\bw^1)} ( |\bx-\bx^1_i|^2-w^1_i ) \rho (\bx) \, d \bx = \int_\Omega \min_i \{ |\bx-\bx^1_i|^2-w^1_i \} \rho (\bx) \, d \bx.
\]

Property (iv): First we check that $\bom \left( \bX^2 , \bw^2 \right)$ is a critical point of the function defined by
$g \left( \bM \right) =  H \left( \left( \bX^1, \bomega(\bX^2,\bw^2) \right),(\bX^2,\bw^2), \bM \right)$:
\begin{equation}
\nonumber
\frac{\partial g}{\partial M_j} \left( \bom \left( \bX^2 , \bw^2 \right) \right)
=
\omega_j \left( \bX^2 , \bw^2 \right) + f'\left( m_j \left( \bX^2 , \bw^2 \right) \right) = 0
\end{equation}
by the definition \eqref{eqn:lloydmap} of $\omega_j$.
Note that the function $g$ is concave since its Hessian is diagonal with non-positive diagonal entries:
\begin{equation}
\nonumber
D^2 g = \textrm{diag} \left( f''(M_1),f''(M_2),\ldots,f''(M_N) \right).
\end{equation}
Therefore the critical point $\bom \left( \bX^2 , \bw^2 \right)$ is a global maximum point of $g$, as required.
\end{proof}

\subsection{Critical points of $E$}
\label{Subsec: Crit points}
In this section we show that critical points of $E$ are fixed points of the Lloyd maps $\bxi$, $\bomega$.

\begin{lemma}[Partial derivatives of $E$]
\label{lem:DE}
The partial derivatives of $E$ are
\begin{align}
\label{eqn:E_x}
\frac{\partial E}{\partial \bx_i}(\bX,\bw)
& = 2 m_i (\bx_i - \bxi_i(\bX,\bw)) + \sum_{j=1}^N \frac{\partial m_j}{\partial \bx_i} (w_j - \omega_j (\bX,\bw)),
\\
\label{eqn:E_w}
\frac{\partial E}{\partial w_i}(\bX,\bw)
& = \sum_{j=1}^N \frac{\partial m_j}{\partial w_i} (w_j - \omega_j (\bX,\bw))
\end{align}
for $i \in \{ 1, \ldots, N \}$. In matrix notation:
\begin{equation}
\label{eq:mform}
\begin{pmatrix}
\nabla_{\bX} E \\ \nabla_{\bw} E
\end{pmatrix}
=
\begin{pmatrix}
2 \hat{\bM} & \nabla_{\bX} \bom \\
\bO & \nabla_{\bw} \bom
\end{pmatrix}
\begin{pmatrix}
\bX - \bxi(\bX,\bw) \\
\bw - \bomega(\bX,\bw)
\end{pmatrix}
\end{equation}
where
\begin{equation}
\label{eq:Mhat}
\hat{\bM} := \mathrm{diag}(m_1,\ldots,m_N) \otimes \bI_{d} =
\mathrm{diag}(m_1 \bI_{d},\ldots, m_N \bI_{d}).
\end{equation}
\end{lemma}

\begin{proof}
From equation \eqref{eqn:E=H},
\begin{equation}
\label{eqn:E_x 2}
\frac{\partial E}{\partial \bx_i}(\bX,\bw) = \frac{\partial H}{\partial \bx^1_i} + \frac{\partial H}{\partial \bx^2_i}
+ \sum_j \frac{\partial H}{\partial M_j} \frac{\partial m_j}{\partial \bx_i}
\end{equation}
where the derivatives of $H$ are evaluated at $((\bX,\bw),(\bX,\bw),\bom(\bX,\bw))$.
The second term on the right-hand side is zero by Lemma \ref{lemma:H}(iii). Direct computation (as in the proof of Lemma \ref{lemma:H}(i),(iv)) gives
\begin{equation}
\label{eqn:H_x1, H_m}
\frac{\partial H}{\partial \bx^1_i} = 2 m_i (\bx_i-\bxi_i), \quad
\frac{\partial H}{\partial M_j} = w_j + f'(m_j(\bX,\bw)).
\end{equation}
Combining \eqref{eqn:E_x 2}, \eqref{eqn:H_x1, H_m} and the definition of $\omega_j$ yields \eqref{eqn:E_x}.

Differentiating \eqref{eqn:E=H} with respect to $w_i$ gives
\begin{equation}
\label{eqn:E_w 2}
\frac{\partial E}{\partial w_i}(\bX,\bw) = \frac{\partial H}{\partial w^1_i} + \frac{\partial H}{\partial w^2_i}
+ \sum_j \frac{\partial H}{\partial M_j} \frac{\partial m_j}{\partial w_i}
\end{equation}
where the derivatives of $H$ are evaluated at $((\bX,\bw),(\bX,\bw),\bom(\bX,\bw))$.
The first two terms on the right-hand side are zero by Lemma \ref{lemma:H}(ii),(iii). Therefore combining \eqref{eqn:E_w 2} and
\eqref{eqn:H_x1, H_m}$_2$ yields \eqref{eqn:E_w}.
\end{proof}

\paragraph{Weighted graph Laplacian matrices}
Given a power diagram $\{P_i (\bX,\bw) \}$ define a graph $G$ that has as vertices $\bX$, and edges given by the neighbour relations of the power diagram: $\bx_i$ is connected by an edge to $\bx_j$ if and only if $i \in J_j$ (and equivalently $j \in J_i$). If we associate a weight $u_{ij}=u_{ji}$ to each edge of this graph, then we can define the weighted graph Laplacian matrix $L= L(G,u)$ by
\begin{equation}
L_{ij} =
\left\{
\begin{array}{cl}
\displaystyle \sum_{k \in J_j} u_{jk} & \textrm{if } i=j, \\
- u_{ij} & \textrm{if } i \in J_j, \\
0 & \textrm{otherwise}.
\end{array}
\right.
\end{equation}
The symmetric matrix $L$ is the difference between the weighted degree matrix and weighted adjacency matrix of $G$. It is well-known that the dimension of the null space of $L$ equals the number of connected components of $G$. See \cite[p.~117, Th.~3.1]{Mohar2004}. In our case $G$ is connected and so, for any edge-weighting $u$, the null space of $L(G,u)$ is one-dimensional and is spanned by $(1,1,\ldots,1)$. In an analogous way, one can define (block) weighted graph Laplacian matrices for vector-valued weights $\bm{u}_{ij}$.

Computing the derivatives of $m_j$ that appear in equations \eqref{eqn:E_x} and \eqref{eqn:E_w} is delicate since this involves differentiating the integrals $m_j = \int_{P_j(\bX,\bw)} \rho \, d \bx$ with respect to $\bx_i$ and $w_i$. It turns out that these derivatives are weighted graph Laplacian matrices:

\begin{lemma}[Weighted graph Laplacian structure of $\nabla_{\bX} \bom$ and $\nabla_{\bw} \bom$]
\label{lem:Dm}
Let $(\bX,\bw) \in \mathcal{G}^N$ be the generators of a power diagram with the generic property that adjacent cells have a common face (a common edge in 2D).
The partial derivatives of $\bom(\bX,\bw)$ are
\begin{align}
\nonumber
\frac{\partial m_j}{\partial \bx_i}
& =
\left\{
\begin{array}{cl}
\displaystyle \sum_{k\in J_j}\frac{m_{jk}}{d_{jk}}\left(\overline{\bx}_{jk}-\bm{x}_j\right) & \textrm{if } i=j,
\vspace{0.1cm}
\\
\displaystyle -\frac{m_{ij}}{d_{ij}}\left(\overline{\bx}_{ij}-\bm{x}_i\right) &  \textrm{if }i \in J_j,
\vspace{0.1cm}
\\
\bO & \textrm{otherwise},
\end{array}
\right.
\\
\nonumber
\frac{\partial m_j}{\partial w_i}
& =
\left\{
\begin{array}{cl}
\displaystyle \sum_{k\in J_j}\frac{m_{jk}}{2d_{jk}} & \textrm{if } i=j,
\vspace{0.1cm}
\\
\displaystyle -\frac{m_{ij}}{2d_{ij}} & \textrm{if } i \in J_j,
\vspace{0.1cm}
\\
0 & \textrm{otherwise},
\end{array}
\right.
\end{align}
for $i \in \{ 1, \ldots, N \}$.
In particular, the $N$-by-$N$ matrix $\nabla_{\bw} \bom$, which has components $[\nabla_{\bw} \bom]_{ij}=\partial m_j / \partial w_i$, is the weighted graph Laplacian matrix of $G(\bX,\bw)$ with respect to the weights $\frac{m_{ij}}{2 d_{ij}}$.  Therefore the null space of $\nabla_{\bw} \bom$ is one-dimensional and is spanned by $(1,1,\ldots,1) \in \mathbb{R}^N$.  Note that $(1,1,\ldots,1)$ also belongs to the null space of the $(Nd)$-by-$N$ matrix $\nabla_{\bX} \bom$, which has $d$-by-$1$ blocks $[\nabla_{\bX} \bom]_{ij}=\partial m_j/ \partial \bx_i$.
\end{lemma}

\begin{proof}
 Given the power diagram $\{ P_j \}_{j=1}^N$ generated by $(\bX,\bw) \in \mathcal{G}^N$, let $\{ P_j^t \}_{j=1}^N$ be the power diagram generated by
 $(\bX^t,\bw^t):=(\bX + t \tilde{\bX},\bw + t \tilde{\bw})$ for some $ \tilde{\bX} \in (\mathbb{R}^d)^N$, $\tilde{\bw} \in \mathbb{R}^N$. For $t$ in a small enough neighbourhood of zero, this family of power diagrams has the same number of cells, and each cell has the same number of faces, as the power diagram generated by $(\bX,\bw)$ (this follows from the assumption that adjacent cells have a common face).
Let $\varphi^t : \Omega \to \Omega$ be any flow map with the properties that $\varphi^0$ is the identity map, $\varphi^t(\bX)=\bX^t$, $\varphi^t(P_j)=P_j^t$ for all $j$, and that $\varphi^t$ maps the faces of $P_j$ to the faces of $P_j^t$ for all $j$.
Fix $j$ and consider
\begin{equation}
\label{eqn:mt}
m_j(\bX^t,\bw^t) = \int_{P_j^t} \rho \, d \bx = \int_{\varphi^t(P_j)} \rho \, d \bx.
\end{equation}
Define $V(\bx) = \frac{d}{dt} \varphi^t(\bx) |_{t=0}$.
By the Reynolds Transport Theorem, differentiating \eqref{eqn:mt} with respect to $t$ and evaluating at $t=0$ gives
\begin{equation}
\label{eqn:deriv}
\sum_{i=1}^N \frac{\partial m_j}{\partial \bx_i} \cdot \tilde{\bx}_i + \frac{\partial m_j}{\partial w_i} \tilde{w}_i
 = \int_{\partial P_j} \rho \, V \cdot \bn \, dS
 = \sum_{k \in J_j} \int_{F_{jk}} \rho \, V \cdot \bn_{jk} \, dS.
\end{equation}
Now we compute $V \cdot \bn_{jk}$.
Choose a face $F_{jk}=P_j \cap P_k$ and some point $\bx \in F_{jk}$. Then $\bx^t := \varphi^t(\bx) \in F^t_{jk} = P_j^t \cap P_k^t$ and so it satisfies
\[
| \bx^t - \bx_j^t |^2 - w_j^t =  | \bx^t - \bx_k^t |^2 - w_k^t.
\]
Differentiating with respect to $t$ and setting $t=0$ gives
\begin{equation}
\label{eqn:bndry}
2 (\bx - \bx_j) \cdot (V(\bx)-\tilde{\bx}_j) - \tilde{w}_j =
2 (\bx - \bx_k) \cdot (V(\bx)-\tilde{\bx}_k) - \tilde{w}_k.
\end{equation}
Recall that $\bn_{jk} = (\bx_k - \bx_j)/d_{jk}$. Therefore rearranging \eqref{eqn:bndry} and dividing by $d_{jk}$ yields
\begin{equation}
\label{eq:Vdotn}
V(\bx) \cdot \bn_{jk} = \frac{ (\bx - \bx_j) \cdot \tilde{\bx}_j - (\bx - \bx_k) \cdot \tilde{\bx}_k}{d_{jk}} + \frac{\tilde{w}_j - \tilde{w}_k}{2 d_{jk}}.
\end{equation}
Substituting this into \eqref{eqn:deriv} and using \eqref{eqn:def2}$_2$ and \eqref{eqn:def3}$_2$ gives
\[
\sum_{i=1}^N \frac{\partial m_j}{\partial \bx_i} \cdot \tilde{\bx}_i + \frac{\partial m_j}{\partial w_i} \tilde{w}_i
= \sum_{k \in J_j}
\frac{ m_{jk}}{d_{jk}} [(\overline{\bx}_{jk} - \bx_j) \cdot \tilde{\bx}_j - (\overline{\bx}_{jk} - \bx_k) \cdot \tilde{\bx}_k ]
+ \frac{m_{jk}}{2 d_{jk}}(\tilde{w}_j - \tilde{w}_k).
\]
The derivatives in Lemma \ref{lem:Dm} can be read off from this equation by making suitable choices of $(\tilde{\bX},\tilde{\bw})$.
\end{proof}

\begin{remark}
\upshape
The fact that $(1,1,\ldots,1)\in \mathbb{R}^N$ belongs to the null space of the matrix $\nabla_{\bw} \bom$ corresponds to the fact that the power diagram has fixed total mass and that it is invariant under the addition of a constant to all its weights:
\begin{equation}
\sum_j m_j = \int_\Omega \rho(\bx) \, d \bx,  \qquad m_j (\bX,\bw + (c,c,\ldots,c)) = m_j (\bX,\bw).
\end{equation}
Differentiating the first equation with respect to $w_i$ gives $\sum_j \partial m_j/\partial w_i = 0$ for all $i$, and so $(1,1,\ldots,1)$ belongs to the null space of $\nabla_{\bw} \bom$.
Differentiating the second equation with respect to
$c$ and then setting $c=0$ gives $\sum_i \partial m_j / \partial w_i = 0$ for all $j$, and so $(1,1,\ldots,1)$ belongs to the null space of $(\nabla_{\bw} \bom)^T$ (which equals $\nabla_{\bw} \bom$ since $\nabla_{\bw} \bom$ is symmetric).
\end{remark}

The main result of this section is the following:
\begin{proposition}[Critical points of $E$ are fixed points of the Lloyd maps]
\label{prop:critE}
Let $(\bX,\bw) \in \mathcal{G}^N$ be a critical point of $E$. Then, up to the addition of a constant to the weights, $(\bX,\bw)$ is a fixed point of the Lloyd maps $\bxi$ and $\bomega$:
\begin{equation}
\label{eq:fp}
\bxi ( \bX,\bw) = \bX, \quad \bomega(\bX,\bw) = \bw + \bc
\end{equation}
where $\bc = c (1,1,\ldots,1) \in \mathbb{R}^N$.
In particular, critical points of $E$ are centroidal power diagrams.
\end{proposition}

\begin{proof}
Equation \eqref{eqn:E_w} yields
\[
\bO = \nabla_{\bw} E = \nabla_{\bw} \bom (\bw-\bomega(\bX,\bw)).
\]
 By Lemma \ref{lem:Dm}, $\bomega(\bX,\bw) = \bw + \bc$ for some $\bc = c (1,1,\ldots,1) \in \mathbb{R}^N$.
 Since $\bc$ belongs to the null space of $\nabla_{\bX} \bom$, then equation \eqref{eqn:E_x} implies that
\begin{equation}
\label{eq:CP}
\bO = \frac{\partial E}{\partial \bx_i}(\bX,\bw) = 2 m_i (\bx_i - \bxi_i(\bX,\bw)).
\end{equation}
 By assumption the power diagram generated by $(\bX,\bw)$ has no empty cells. Therefore $m_i \ne 0$ for any $i$ and equation \eqref{eq:CP} gives $\bX -\bxi ( \bX,\bw) = \bO$, as required.
\end{proof}

\begin{remark}[Examples of critical points of $E$]
\upshape
Any centroidal Voronoi tessellation of $\Omega$ with the property that all cells have the same mass is a critical point of $E$. If $\rho=$ constant and $\Omega$ is a domain with nice symmetry, e.g., a square or a disc, then it is easy to write down lots, in fact infinitely many, centroidal Voronoi tessellations with this property and hence find infinitely many critical points of $E$ (although not all will be local minima). The highly non-convex nature of the energy landscape makes it difficult to find global minima. See \S\ref{subsec:nonconvex}.
\end{remark}


\section{Properties of the algorithm}
\label{sec:prop}

Our main result is the following:

\begin{theorem}
\label{thm:energyDecrease}
The generalized Lloyd algorithm is energy decreasing:
\begin{equation}
\nonumber
E( \bX^{n+1}, \bw^{n+1} ) \le E( \bX^{n}, \bw^{n} )
\end{equation}
where $\bX^{n+1}=\bxi \left( \bX^n, \bw^n\right)$, $\bw^{n+1}=\bomega \left( \bX^n, \bw^n\right)$, $(\bX^n,\bw^n)\in\mathcal{G}^N$.
The inequality is strict unless $(\bX^{n+1}, \bw^{n+1}) = (\bX^{n+2}, \bw^{n+2})$, i.e., unless the algorithm has converged.
\end{theorem}
\begin{proof}
The proof follows easily by stringing together the properties of $H$ from Lemma \ref{lemma:H}:
\begin{align*}
& E \left( \bX^{n}, \bw^{n}  \right)
\\
& = H \left( \left( \bX^{n}, \bw^{n+1} \right) , \left( \bX^{n}, \bw^{n} \right)  , \bom \left( \bX^{n}, \bw^{n} \right) \right)
& \textrm{(by Lemma \ref{lemma:H}(ii))}
\\
& = H \left( \left( \bX^{n},\bomega \left( \bX^{n}, \bw^{n} \right) \right) ,  \left( \bX^{n}, \bw^{n} \right) , \bom \left( \bX^{n}, \bw^{n} \right) \right)
& \textrm{(by definition of } \bw^{n+1} \textrm{)}
\\
& \ge H \left( \left( \bX^{n},\bomega \left( \bX^{n}, \bw^{n} \right) \right) ,  \left( \bX^{n}, \bw^{n} \right) , \bom \left( \bX^{n+1}, \bw^{n+1} \right) \right)
& \textrm{(by Lemma \ref{lemma:H}(iv))}
\\
& = H \left( \left( \bX^{n},\bw^{n+1} \right) ,  \left( \bX^{n}, \bw^{n} \right) , \bom \left( \bX^{n+1}, \bw^{n+1} \right) \right)
& \textrm{(by definition of } \bw^{n+1} \textrm{)}
\\
& \ge H \left( \left( \bxi \left( \bX^{n},\bw^{n} \right),\bw^{n+1} \right) ,  \left( \bX^{n}, \bw^{n} \right) , \bom \left( \bX^{n+1}, \bw^{n+1} \right) \right)
& \textrm{(by Lemma \ref{lemma:H}(i))}
\\
& = H \left( \left( \bX^{n+1},\bw^{n+1} \right) ,  \left( \bX^{n}, \bw^{n} \right) , \bom \left( \bX^{n+1}, \bw^{n+1} \right) \right)
& \textrm{(by definition of } \bX^{n+1} \textrm{)}
\\
& \ge H \left( \left( \bX^{n+1},\bw^{n+1} \right) ,  \left( \bX^{n+1}, \bw^{n+1} \right) , \bom \left( \bX^{n+1}, \bw^{n+1} \right) \right)
& \textrm{(by Lemma \ref{lemma:H}(iii))}
\\
& = E \left( \bX^{n+1}, \bw^{n+1} \right)
& \textrm{(by equation \eqref{eqn:E=H})}.
\end{align*}
By Lemma \ref{lemma:H}(iii) the last inequality is strict unless $P_i \left( \bX^{n+1}, \bw^{n+1} \right)= P_i \left( \bX^{n}, \bw^{n} \right)$ for all
$i$, up to sets of $\rho \, d \bx$--measure zero, in which case $x_i^{n+2}$ (which is the centroid of $P_i( \bX^{n+1}, \bw^{n+1})$) equals $x_i^{n+1}$ (which is the centroid of $P_i( \bX^{n}, \bw^{n})$)
and
\begin{equation}
\nonumber
w_i^{n+2}=-f'(|P_i( \bX^{n+1}, \bw^{n+1})|) = -f'(|P_i \left( \bX^{n}, \bw^{n} \right)|) = w_i^{n+1}
\end{equation}
 as required.
\end{proof}

\begin{remark}[Elimination of generators is energy decreasing]
\upshape
The generalized Lloyd algorithm removes generators corresponding to empty cells, i.e., if $P^n_i=\emptyset$, then
the generator pair $(\bx^n_i,w^n_i)$ is removed in Step (2) of Algorithm \ref{algo:genLloyd}. The assumption that $f(0) \ge 0$ ensures that removing generators is energy decreasing.
\end{remark}

Recall from equation \eqref{eqn:G} that $\mathcal{G}^N$ is the set of $N$ generators such that no two generators
coincide and that the corresponding power diagram has no empty cells.
The energy-decreasing property of the algorithm can be used to prove the following convergence result,
which is a generalization of convergence theorem for the classical Lloyd algorithm \cite[Thm.~2.6]{Du2006}:
\begin{theorem}[Convergence of the generalized Lloyd algorithm]
\label{thm:conv}
Assume that $E$ has only finitely many critical points with the same energy.
Let $(\bX^k,\bw^k)$ be a sequence generated by Algorithm \ref{algo:genLloyd}. Let $K$ be large enough such that, for all $k \ge K$,
$(\bX^k,\bw^k) \in \mathcal{G}^N$ for $N$ fixed, i.e., there is no elimination of generators after iteration $K$.
If the sequence $(\bX^k,\bw^k)_{k > K}$ is a compact subset of $\mathcal{G}^N$, then it converges to a critical point of $E$.
\end{theorem}
\begin{proof}
This follows by combining a minor modification of the proof of the Global Convergence Theorem from \cite[p.~206]{LuenbergerYe} with a convergence theorem for the classical Lloyd algorithm \cite[Thm.~2.5]{Du2006}.
Note that the Lloyd maps $\bxi_i$, $\bomega_i$ and the energy $E$ are continuous on $\mathcal{G}^N$ by the continuity of the mass and first and second moments of mass of the power cells $P_i$,
and the continuity of $f$.

Let $(\bX^{k_j},\bw^{k_j})$ be a convergent subsequence converging to $(\bX,\bw)\in \mathcal{G}^N$. By the continuity of $E$ on $\mathcal{G}^N$, $E(\bX^{k_j},\bw^{k_j}) \to E(\bX,\bw)$. Take $J$ large enough so that $E(\bX^{k_J},\bw^{k_J}) - E(\bX,\bw) < \varepsilon$.
By Theorem \ref{thm:energyDecrease} the whole sequence $E(\bX^{k},\bw^{k})$ converges to $E(\bX,\bw)$ since for all $k > k_J$
\[
0 \le E(\bX^k,\bw^k) - E(\bX,\bw)
\le  E(\bX^k,\bw^k) - E(\bX^{k_J},\bw^{k_J}) + E(\bX^{k_J},\bw^{k_J}) - E(\bX,\bw) < \varepsilon.
\]

Next we check that $(\bX,\bw)$ is a fixed point of the Lloyd maps and hence a critical point of $E$. Consider the sequence
$(\bX^{k_j-1},\bw^{k_j-1})$. By the compactness of $(\bX^{k},\bw^{k})$ there is a subsequence $(\bX^{k_{j_l}-1},\bw^{k_{j_l}-1})$
converging to $(\bX_-,\bw_-) \in \mathcal{G}^N$. The continuity of the Lloyd maps on $\mathcal{G}^N$ implies that
\[
(\bxi(\bX^{k_{j_l}-1},\bw^{k_{j_l}-1}),\bomega(\bX^{k_{j_l}-1},\bw^{k_{j_l}-1}))=(\bX^{k_{j_l}},\bw^{k_{j_l}}) \to
(\bxi(\bX_-,\bw_-),\bomega(\bX_-,\bw_-)).
\]
But $(\bX^{k_{j_l}},\bw^{k_{j_l}}) \to (\bX,\bw)$. Therefore
$(\bxi(\bX_-,\bw_-),\bomega(\bX_-,\bw_-)) = (\bX,\bw)$.
Since $E(\bX^{k},\bw^{k}) \to E(\bX,\bw)$, we obtain that
\[
E(\bX_-,\bw_-) = E(\bX,\bw) = E(\bxi(\bX_-,\bw_-),\bomega(\bX_-,\bw_-))
\]
and thus, by Theorem \ref{thm:energyDecrease}, $(\bxi(\bX_-,\bw_-),\bomega(\bX_-,\bw_-))=(\bX,\bw)$
is a fixed point of the Lloyd maps.

We have shown that any accumulation point of $(\bX^k,\bw^k)$ is a fixed point of the Lloyd maps and, by the energy-decreasing property of the algorithm,
all accumulation points have the same energy. Therefore, by the first assumption of the theorem, it follows that
$(\bX^k,\bw^k)$ has only finitely many accumulation points.

Finally, the whole sequence $(\bX^k,\bw^k)$ converges to $(\bX,\bw)$ by the following result, which is proved in \cite[Thm.~2.5]{Du2006}
for the classical Lloyd algorithm but holds for general fixed point methods of the form $z^{k+1} = T(z^k)$: If the sequence $\{z^k\}$ generated by $z^{k+1} = T(z^k)$ has finitely many accumulation points, $T$ is continuous
at them, and they are fixed points of $T$, then $z^k$ converges. This completes the proof.
\end{proof}
\begin{remark}[Assumptions of the convergence theorem]
\upshape
The assumption that $E$ has only finitely many critical points with the same energy is true for generic domains $\Omega$
but not for all, e.g., if $\Omega$ is a ball and $\rho$ is radially symmetric then there could be infinitely many fixed points
with the same energy by rotational symmetry. The assumption that $(\bX^k,\bw^k)_{k > K}$ is a compact subset of $\mathcal{G}^N$ is stronger.
It means that in the limit there is no elimination of generators.
We need this assumption since the Lloyd maps are not defined if there are empty cells, $P_i = \emptyset$ for some $i$.
While numerical experiments suggest that cells do not
disappear in the limit, it is difficult to prove, even for the classical Lloyd algorithm; it was proved in one-dimension
by \cite[Prop.~2.9]{Du2006}. For further convergence theorems for the classical Lloyd algorithm see \cite{Du1999} and \cite{SabinGray}.
\end{remark}

\begin{remark}[Interpretation of the Lloyd algorithm as a descent method]
\upshape
In the following proposition we study the structure of the generalized Lloyd algorithm. Recall that an iterative method is a
descent method for an energy $\mathcal{E}$ if it can be written in the form
\begin{equation}
\label{eqn:descent}
\bz_{n+1} = \bz_n - \alpha_n \bB_n \nabla \mathcal{E}
\end{equation}
where $\bB_n$ is positive-definite, $\alpha_n$ is the step size, and $-\bB_n \nabla \mathcal{E}$ is the step direction, e.g., $\bB_n = \bI$ is the steepest descent method, $B_n = (D^2\mathcal{E})^{-1}$, $\alpha_n=1$ is Newton's method.
The following proposition asserts that the generalized Lloyd algorithm can be written in the form \eqref{eqn:descent}, but not that
$\bB_n$ is positive-definite, which we are unable to prove:
\end{remark}

\begin{proposition}
\label{prop:DM}
The generalized Lloyd algorithm can be written in the form
\begin{equation}
\label{eq:DM}
\begin{pmatrix}
\bX^{n+1} \\ \bw^{n+1}
\end{pmatrix}
=
\begin{pmatrix}
\bX^{n} \\ \bw^{n}
\end{pmatrix}
- \bB_n
\begin{pmatrix}
\nabla_{\bX} E^n \\ \nabla_{\bw} E^n
\end{pmatrix}
+
\begin{pmatrix}
\bO \\ \bc
\end{pmatrix}
\end{equation}
where $\bB_n$ is a square matrix of dimension $N(d+1)$ 
and $\bc = c (1,1,\ldots,1)^T$ for some $c \in \mathbb{R}$.
\end{proposition}
\begin{proof}
Recall that
\[
m_i^n = \int_{P_i(\bX^n,\bw^n)} \rho(\bx) \, d \bx.
\]
and $\hat{\bM}_n = \textrm{diag}(m_1^n \bI_{d},\ldots, m_N^n \bI_{d})$.
Equation \eqref{eq:mform} implies that
\begin{equation}
\label{eq:2invert}
\begin{pmatrix}
\nabla_{\bX} E^n \\ \nabla_{\bw} E^n
\end{pmatrix}
=
\begin{pmatrix}
2 \hat{\bM}_n & \nabla_{\bX} \bom^n \\
\bO & \nabla_{\bw} \bom^n
\end{pmatrix}
\begin{pmatrix}
\bX^n - \bX^{n+1} \\
\bw^n - \bw^{n+1}
\end{pmatrix},
\end{equation}
where $\bO$ is the $N$-by-$(Nd)$ zero matrix.
By Lemma \ref{lem:Dm}, the matrix on the right-hand side has a one-dimensional nullspace. Therefore rewriting these equations in the form \eqref{eq:DM} requires some care.

Let $\be_1, \ldots, \be_N$ be the standard basis vectors for $\mathbb{R}^N$. We introduce the new basis
\[
\bof_1 := \be_1 - \be_2, \quad \bof_2 := \be_2 - \be_3, \quad \ldots \quad \bof_{N-1} := \be_{N-1} - \be_N, \quad \bof_N := \be_1 + \cdots + \be_N.
\]
Note that $\bof_N$ spans the null space of $\nabla_{\bw} \bom^n$. Let $P$ be the invertible change-of-basis matrix satisfying $P \bof_i = \be_i$. In particular
\[
P^{-1} =
\begin{pmatrix}
\phantom{-}1 & & & & & 1 \\
-1 & \phantom{-}1 & & & & 1 \\
& -1 & \phantom{-}1 & & & 1 \\
& & \ddots & \ddots & & \vdots \\
& & & \ddots & \ddots & \vdots \\
& & & & -1 & 1
\end{pmatrix}
\]
with zeros where no entry is given.
Let $\Pi:\mathbb{R}^N \to \mathbb{R}^{N-1}$ be the projection onto $\{\bof_N\}^\perp$:
\[
\Pi = ( \bI_{N-1} | \bO )
\]
where $\bO$ is the $(N-1)$-by-$1$ zero vector.
Observer that for all $\by \in \mathbb{R}^N$
\begin{equation}
\label{eq:mess}
\nabla_{\bw} \bom^n \, \by =  \nabla_{\bw} \bom^n \, P^{-1} \Pi^T \Pi P \, \by
\end{equation}
since $\nabla_{\bw} \bom^n \, \bof_N = \bO$, $\Pi P \, \bof_N = \bO$, and
$\Pi^T \Pi \, \be_i = \be_i$ for all $i \in \{ 1, \ldots, N-1 \}$.
We check that the following $(N-1)$-by-$(N-1)$ matrix is invertible:
\begin{equation}
\label{eq:An}
A_n := \Pi P \, \nabla_{\bw} \bom^n P^{-1} \Pi^T.
\end{equation}
If $A_n \bx = \bO$, then $P \, \nabla_{\bw} \bom^n P^{-1} \Pi^T \, \bx = c \,\be_N$ for some $c \in \mathbb{R}$, and so
$\nabla_{\bw} \bom^n P^{-1} \Pi^T \, \bx = c \bof_N$. But
$\bof_N^T \nabla_{\bw}\bom^n = (\nabla_{\bw}\bom^n \bof_N)^T = \bO$
and thus $c=0$.
Therefore, by Lemma \ref{lem:Dm}, $P^{-1} \Pi^T \, \bx = a \bof_N$ for some $a \in \mathbb{R}$. It follows from the definitions of $P$ and $\Pi$ that $a=0$ and $\bx=\bO$.

Using equations \eqref{eq:mess} and \eqref{eq:An}, we see that the equation
\[
\nabla_{\bw} E^n = \nabla_{\bw} \bom^n (\bw^n - \bw^{n+1})
\]
can be inverted to give
\[
A_n^{-1} \Pi P \, \nabla_{\bw} E^n = \Pi P (\bw^n - \bw^{n+1}).
\]
Therefore
\begin{equation}
\label{eq:wnp1}
\bw^{n+1} = \bw^n - P^{-1} \Pi^T A_n^{-1} \Pi P \, \nabla_{\bw} E^n + c \bof_N
\end{equation}
for some $c \in \mathbb{R}$.
We conclude from equations \eqref{eq:2invert} and \eqref{eq:wnp1} that
\[
\begin{pmatrix}
\bX^{n+1} \\ \bw^{n+1}
\end{pmatrix}
=
\begin{pmatrix}
\bX^{n} \\ \bw^{n}
\end{pmatrix}
- \bB_n
\begin{pmatrix}
\nabla_{\bX} E^n \\ \nabla_{\bw} E^n
\end{pmatrix}
+
\begin{pmatrix}
\bO \\ c \bof_N
\end{pmatrix}
\]
where $\bB_n$ is the matrix
\[
\bB_n =
\begin{pmatrix}
\frac12 \hat{\bM}_n^{-1} & -\frac 12 \hat{\bM}_n^{-1} \nabla_{\bX} \bom^n P^{-1} \Pi^T A_n^{-1} \Pi P \\
\bO & P^{-1} \Pi^T A_n^{-1} \Pi P
\end{pmatrix},
\]
where $\bO$ is the $N$-by-$(Nd)$ zero matrix. This completes the proof.
\end{proof}


\begin{remark}[Alternative algorithm]
\upshape
The following proposition gives explicit expressions for the derivatives of the Lloyd maps $\bxi$ and $\bomega$. These could
be used to find critical points of $E$ in an alternative way, e.g., by solving the nonlinear equations
\eqref{eq:fp} using Newton's method.
\end{remark}

\begin{proposition}[Derivatives of the Lloyd maps]
\label{prop:DLloyd}
Given a face $F$ of a power diagram, define the matrix $\mathcal{S}(F)$ by
\[
\mathcal{S}(F) = \frac{1}{m(F)} \int_F \bx \otimes \bx \, \rho(\bx) \, dS
\]
where $m(F)=\int_F \rho \, dS$ is the mass of the face.
Let $(\bX,\bw) \in \mathcal{G}^N$ be the generators of a power diagram with the generic property that adjacent cells have a common face (a common edge in 2D).
The derivatives of the Lloyd maps $\bxi(\bX,\bw)$ and $\bomega(\bX,\bw)$ are 
\[
\left( \frac{\partial \bxi}{\partial \bX} \right)_{ij} = \frac{\partial \bxi_i}{\partial \bx_j} =
\left\{
\begin{array}{cl}
\displaystyle
\frac{1}{m_i} \sum_{k \in J_i} \frac{m_{ik}}{d_{ik}}
(\mathcal{S}(F_{ik}) - \overline{\bx}_{ik} \otimes \bx_i + \overline{\bx}_i \otimes (\bx_i - \overline{\bx}_{ik}))
& \textrm{if } i=j,
\vspace{0.1cm}
\\
\displaystyle
- \frac{m_{ij}}{m_i d_{ij}} (\mathcal{S}(F_{ij}) - \overline{\bx}_{ij} \otimes \bx_j + \overline{\bx}_i \otimes (\bx_j - \overline{\bx}_{ij}))
& \textrm{if } j \in J_i,
\vspace{0.1cm}
\\
\bO & \textrm{otherwise},
\end{array}
\right.
\]
\[
\left( \frac{\partial \bxi}{\partial \bw} \right)_{ij} = \frac{\partial \bxi_i}{\partial w_j} =
\left\{
\begin{array}{cl}
\displaystyle
\frac{1}{2 m_i} \sum_{k \in J_i} \frac{m_{ik}}{d_{ik}} ( \overline{\bx}_{ik}-\overline{\bx}_i )
& \textrm{if } i=j,
\vspace{0.1cm}
\\
\displaystyle
- \frac{m_{ij}}{2 m_i d_{ij}} (\overline{\bx}_{ij}-\overline{\bx}_i ) & \textrm{if } j \in J_i,
\vspace{0.1cm}
\\
\bO & \textrm{otherwise},
\end{array}
\right.
\]
\[
\left( \frac{\partial \bomega}{\partial \bX} \right)_{ij} = \frac{\partial \omega_i}{\partial \bx_j} =
\left\{
\begin{array}{cl}
\displaystyle -f''(m_i) \sum_{k\in J_i}\frac{m_{ik}}{d_{ik}}\left(\overline{\bx}_{ik}-\bm{x}_i\right) & \textrm{if } i=j,
\vspace{0.1cm}
\\
\displaystyle f''(m_i) \frac{m_{ij}}{d_{ij}}\left(\overline{\bx}_{ij}-\bm{x}_i\right) &  \textrm{if }j \in J_i,
\vspace{0.1cm}
\\
\bO & \textrm{otherwise},
\end{array}
\right.
\]
\[
\left( \frac{\partial \bomega}{\partial \bw} \right)_{ij} = \frac{\partial \omega_i}{\partial w_j} =
\left\{
\begin{array}{cl}
\displaystyle -f''(m_i) \sum_{k\in J_i}\frac{m_{ik}}{2d_{ik}} & \textrm{if } i=j,
\vspace{0.1cm}
\\
\displaystyle f''(m_i)\frac{m_{ij}}{2d_{ij}} & \textrm{if } j \in J_i,
\vspace{0.1cm}
\\
0 & \textrm{otherwise}.
\end{array}
\right.
\]
We order the block matrices $\partial \bxi / \partial \bX$, $\partial \bxi / \partial \bw$, $\partial \bomega / \partial \bX$ and $\partial \bomega / \partial \bw$
so that they have dimensions $(Nd)$-by-$(Nd)$, $(Nd)$-by-$N$, $N$-by-$(Nd)$ and $N$-by-$N$.
\end{proposition}


\begin{proof}
Since $\omega_i = -f'(m_i)$, then the partial derivatives of $\bomega$ are obtained immediately from Lemma \ref{lem:Dm}.
Obtaining the partial derivatives of $\bxi = \tfrac{1}{m_i} \int_{P_i} \bx \rho \, dx$ requires a bit more work.
Observe that 
\begin{equation}
\label{eq:Dxi}
\frac{\partial \bxi_i}{\partial \bx_j}=\frac{1}{m_i}\left( \frac{\partial (m_i \bxi_i)}{\partial \bx_j}- \bxi_i \otimes \frac{\partial m_i}{\partial \bx_j}  \right),
\quad
\frac{\partial \bxi_i}{\partial w_j}=\frac{1}{m_i}\left( \frac{\partial (m_i \bxi_i)}{\partial w_j}- \frac{\partial m_i}{\partial w_j} \bxi_i \right).
\end{equation}
Lemma \ref{lem:Dm} gives $\partial m_i / \partial \bx_j$, $\partial m_i / \partial w_j$ and so we just need to compute
 $\partial(m_i \bxi_i)/\partial \bx_j$, $\partial(m_i \bxi_i)/\partial w_j$, i.e., compute the partial derivatives of
\[
(m_i \bxi_i)(\bX,\bw) = \int_{P_i(\bX,\bw)} \bx \rho(\bx) \, d \bx.
\]
The computation is similar to the proof of Lemma \ref{lem:Dm} and so we just sketch the details.
Consider the same 1-parameter family of power diagrams used in the proof of Lemma \ref{lem:Dm}: $\{ P_i^t \} = \{ \varphi^t(P_i) \}$. As
for equation \eqref{eqn:deriv},
\[
\left. \frac{d}{dt} \right|_{t=0} (m_i \bxi_i)(\bX^t,\bw^t)
= \sum_{j=1}^N \frac{\partial (m_i \bxi_i)}{\partial \bx_j} \tilde{\bx}_j + \frac{\partial (m_i \bxi_i)}{\partial w_j} \tilde{w}_j
= \sum_{k \in J_i} \int_{F_{ik}} \bx \rho(\bx) V \cdot \bn_{ik} \, d S
\]
where $V(\bx) = \frac{d}{dt} \varphi^t(\bx) |_{t=0}$.
Combining this with equation \eqref{eq:Vdotn} gives
\begin{align}
\nonumber
& \sum_{j=1}^N \frac{\partial (m_i \bxi_i)}{\partial \bx_j} \tilde{\bx}_j + \frac{\partial (m_i \bxi_i)}{\partial w_j} \tilde{w}_j
\\
\nonumber
& =
\sum_{k \in J_i} \int_{F_{ik}} \bx \rho(\bx)
\left[ \frac{ (\bx - \bx_i) \cdot \tilde{\bx}_i - (\bx - \bx_k) \cdot \tilde{\bx}_k}{d_{ik}} + \frac{\tilde{w}_i - \tilde{w}_k}{2 d_{ik}} \right]
\, dS
\\
\label{eq:T}
& = \sum_{k \in J_i} \frac{m_{ik}}{d_{ik}} \left[ (\mathcal{S}(F_{ik})-\overline{\bx}_{ik} \otimes \bx_i) \tilde{\bx}_i
- (\mathcal{S}(F_{ik})-\overline{\bx}_{ik} \otimes \bx_k) \tilde{\bx}_k + \frac{\overline{x}_{ik}(\tilde{w}_i - \tilde{w}_k)}{2}
\right]
\end{align}
where the matrix $\mathcal{S}(F_{ik})$ was defined in the statement of the proposition.
By combining equations \eqref{eq:Dxi} and \eqref{eq:T} (with suitable choices of $\tilde{\bX}$ and $\tilde{\bw}$) and Lemma \ref{lem:Dm} we obtain the desired expressions
for $\partial \bxi_i / \partial \bx_j$ and $\partial \bxi_i / \partial w_j$.
\end{proof}

Potentially these derivatives could also be used to prove convergence of the Lloyd algorithm by proving that the Lloyd map pair $(\bxi,\bomega):\mathcal{G}^N \to \mathcal{G}^N$ is a contraction. These derivatives are also needed to evaluate the Hessian of $E$, which can be used to check the stability of fixed points:

\begin{proposition}[The Hessian of $E$ evaluated at fixed points]
\label{prop:D2E}
If $(\bX,\bw)$ is a fixed point of the Lloyd maps $\bxi$ and $\bomega$, i.e., if it satisfies equation \eqref{eq:fp}, then the Hessian of $E$
evaluated at $(\bX,\bw)$ is
\[
\begin{pmatrix}
E_{\bX \bX} & E_{\bX \bw} \\ E_{\bw \bX} & E_{\bw \bw}
\end{pmatrix}
=
\begin{pmatrix}
2 \hat{\bM} & \nabla_{\bX} \bom \\
\bO & \nabla_{\bw} \bom
\end{pmatrix}
\begin{pmatrix}
\bI_{Nd} - \frac{\partial \bxi}{\partial \bX} & - \frac{\partial \bxi}{\partial \bw} \\
- \frac{\partial \bomega}{\partial \bX} & \bI_{N} - \frac{\partial \bomega}{\partial \bw}
\end{pmatrix},
\]
where $\bO$ is the $N$-by-$(Nd)$ zero matrix, $\hat{\bM}$ was defined in equation \eqref{eq:Mhat},
$E_{\bX \bX}$ is the $(Nd)$-by-$(Nd)$ block matrix with $d$-by-$d$ blocks $\partial^2 E/\partial \bx_i \partial \bx_j$,
$E_{\bw \bw}$ is the $N$-by-$N$ matrix with entries $[E_{\bw \bw}]_{ij}=\partial^2 E/\partial w_i \partial w_j$,
$E_{\bX \bw}$ is the $(Nd)$-by-$N$ block matrix with $d$-by-$1$ blocks $\partial^2 E/\partial \bx_i \partial w_j$, and
$E_{\bw \bX}$ is the $N$-by-$(Nd)$ block matrix with $1$-by-$d$ blocks $\partial^2 E/\partial w_i \partial  \bx_j$.
\end{proposition}
\begin{proof}
This follows immediately from equation \eqref{eq:mform}.
\end{proof}


To evaluate the Hessian of $E$ at an arbitrary point, rather than just at a fixed point, requires the computation of the Hessian of $\bom$, which is
a rather painful computation that we choose not to do.



\section{Implementation}
\label{sec:implement}
The generalized Lloyd algorithm relies upon the computation of power diagrams. In this section we briefly review different methods for the calculation of the power diagram given a domain $\Omega$ and generators $\{ \bx_i,w_i \}_{i=1}^N$.

\subsection{Half-plane intersection}\label{sec:half-plane}
Recall that $F_{ij}=F_{ji}=P_i\cap P_j$ is the boundary between power cells $P_i$ and $P_j$. Assume that $P_i\cap P_j\neq\emptyset$ and take two distinct points $\bm{x}$ and $\bm{y}$ in $F_{ij}$. By the definition \eqref{eq:pd} of the cells $P_i$ and $P_j$ we have $\left|\bm{x}-\bm{x}_i\right|^2-w_i=\left|\bm{x}-\bm{x}_j\right|^2-w_j$ and $\left|\bm{y}-\bm{x}_i\right|^2-w_i=\left|\bm{y}-\bm{x}_j\right|^2-w_j$. Subtracting leaves
\[
\left(\bm{x}-\bm{y}\right)\cdot\left(\bm{x}_i-\bm{x}_j\right)=0.
\]
This establishes that boundaries between cells are planes with the normal to $F_{ij}$ parallel to $\bm{x}_i-\bm{x}_j$.
A point on the plane can be found by writing $\bm{p}=\bm{x}_i+s\left(\bm{x}_j-\bm{x}_i\right)$ and noting that $\bm{p}\in F_{ij}$ implies
\[
\left|\bm{p}-\bm{x}_i\right|^2-w_i=\left|\bm{p}-\bm{x}_j\right|^2-w_j
\]
from which we deduce
\[
s=\frac{1}{2}+\frac{w_i-w_j}{2\left|\bm{x}_i-\bm{x}_j\right|^2},\quad\bm{p}=\frac{1}{2}\left(\bm{x}_i+\bm{x}_j\right)-\frac{\left(w_j-w_i\right)}{2\left|\bm{x}_j-\bm{x}_i\right|^2}\left(\bm{x}_j-\bm{x}_i\right).
\]
If we define the \emph{half-plane}
\begin{equation}
\label{eqn:halfplane}
H_{ij}=H\left(\bm{x}_i,w_i,\bm{x}_j,w_j\right)=\{\bm{x}\;:\;\|\bm{x}-\bm{x}_i\|^2-w_i\le\|\bm{x}-\bm{x}_j\|^2-w_j\}
\end{equation}
then
\[
P_i=\bigcap_{j=1,j\neq i}^{j=N} H\left(\bm{x}_i,w_i,\bm{x}_j,w_j\right).
\]
The observation that power cells can be expressed as the intersection of half-planes, and the explicit expressions for both a point on the plane and the normal to the plane, is the basis for the \emph{half-plane} method for the computation of a power-diagram \cite{Okabe2000}. The power cell is built iteratively according to Algorithm \ref{algo:halfplane}.
\begin{algorithm}
\caption{The half-plane intersection method, \cite{Okabe2000}.}
\label{algo:halfplane}
\begin{algorithmic}
\REQUIRE{The set $\Omega$ is a convex polyhedron with $n_\Omega$ faces, and there are $N$ generators $\{\bm{x}_i,w_i\}_{i=1}^N$.}
\FOR{Generator $(\bm{x}_i,w_i)$}
\STATE{$\tilde{P}_i=\Omega$}
\FOR{Generators $(\bm{x}_j,w_j)$, $j\neq i$}
\STATE{Calculate $H_{ij}$, given by \eqref{eqn:halfplane}}
\STATE{$\tilde{P}_i\leftarrow\tilde{P_i}\cap H_{ij}$}
\ENDFOR
\STATE{$P_i\leftarrow\tilde{P}_i$}
\RETURN{Power cell $P_i$}
\ENDFOR
\RETURN{The power diagram composed of at most $N$ power cells, $\{P_i\}$}
\end{algorithmic}
\end{algorithm}
The na\"ive half-plane method sets the cell $\tilde{P}_i=\Omega$ initially, and following repeated intersections with half-planes $H_{ij}$ forms the power cell $P_i$. As discussed in \cite{Okabe2000} for Voronoi diagrams, the construction of each cell requires $N-1$ half-plane intersections and the number of operations in each intersection depends upon the number of faces of the cell $\tilde{P}_i$ (we must check whether the boundary of the new half-plane intersects with any of the faces of $\tilde{P}_i$). At worst, each half-plane intersection increases the number of faces by $1$. If initially the cell has $n_\Omega$ faces, then the total number of checks is at most $n_\Omega+\left(n_\Omega+1\right)+\ldots+(n_\Omega+(N-2))=(N-2)n_\Omega+(N-1)(N-2)/2=O(N^2)$. The intersections must be performed to create each cell so the overall time complexity of this method is \emph{at worst} $O(N^3)$ and is usually $O(N^2)$.

Once the power cells are obtained, the centroid and the mass of cell $P_i$ can be determined by quadrature, or in the special case of constant $\rho$ can be calculated explicitly given the vertices of the cell (see Appendix \ref{append:poly}). These quantities are needed to evaluate the energy and to perform a step of the generalized Lloyd algorithm.

\subsection{Lifting method}\label{sec:lift}

A faster method for the computation of the power diagram is given in \cite{Aurenhammer1987}, in which the generators $\{\bm{x}_i,w_i\}_{i=1}^N$ are lifted into $\mathbb{R}^{d+1}$. Given a generator $\left(\bx,w\right)$, where $\bx$ has components $x_j$, $j=1,\ldots,d$, the lifted generator is the vector in $\mathbb{R}^{d+1}$ with components $(x_1,x_2,\ldots, x_d,z)$ where $z=|\bx|^2-w$. In the power diagram computation the \emph{lower convex hull} of the lifted generators is found, giving rise to a regular triangulation of the generators. The $j$-faces of the triangulation (for example in two dimensions the $0$-faces are the generators, the $1$-faces are the edges and the $2$-faces are the triangles) are then transformed into $(d-j)$-faces via a \emph{polar map}. The result of this is that the triangulation formed by the lower convex hull of the lifted generators is transformed into the power diagram based on the generators. The expensive step in this calculation is the calculation of the lower convex hull of a set of points in $\mathbb{R}^{d+1}$. When $d=2$ then convex hull algorithms with complexity $O(N\log N)$ can be used.

\subsection{Other implementation issues}\label{sec:otherimp}

It is worth noting that Algorithm \ref{algo:genLloyd} converges to local minima and a strategy must be adopted to find global minima. In our simulations
we start with a large number of random initial configurations, apply Algorithm \ref{algo:genLloyd} and periodically sort the results. We then continue using Algorithm \ref{algo:genLloyd} on a subset of configurations that are the lowest energy states. In this way we search for global minima, although we cannot guarantee to find them with this heuristic method.

When using constant $\rho$ the results of Appendix \ref{append:poly} allow fast computation of the integrals required. When using non-constant $\rho$ we employ quadrature: the cells are triangulated and each triangle mapped to a reference triangle on which an $N$-point (we use $N=31$) quadrature rule is applied.




\section{Illustrations and Applications}
\label{sec:illus}
In this section we implement the algorithm in two and three dimensions. We use crystallization and optimal location problems to illustrate the typical flatness and non-convexity of the energy landscape and the rate of convergence of the algorithm. We finish in \S\ref{Subsec:3D} with a more serious application, where we use the algorithm to test a conjecture about the optimality of the BCC lattice for a crystallization problem in three dimensions.

\subsection{Non-convexity and flatness of energy landscape}
\label{subsec:nonconvex}
In this section we look for critical points of the two-dimensional block copolymer energy from \S\ref{Subsubsec: block copolymer}:
\begin{equation}
\label{eqn:bc}
E \left( \{\bx_i,w_i\} \right) = \sum_{i=1}^N \left\{ \lambda  \sqrt{m_i} + \int_{P_i} |\bx-\bx_i|^2 \,d \bx \right\}
\end{equation}
where $\bx_i \in \Omega = [0,1]^2$.
This example first appeared in \cite{BournePeletierRoper}.
It is the special case of \eqref{eqn:energy} with $\rho=1$, $f(m)=\lambda \sqrt{m}$, where $\lambda > 0$ is a parameter representing the strength of the repulsion between the two phases of the block copolymer. The scaling of the energy suggests that the optimal value of $N$ scales like $\lambda^{-\frac{2}{3}}$. Figure \ref{fig:flatness} shows local minimizers of $E$ for $\lambda=0.005$.
\begin{figure}[h!]
\includegraphics[width=\textwidth]{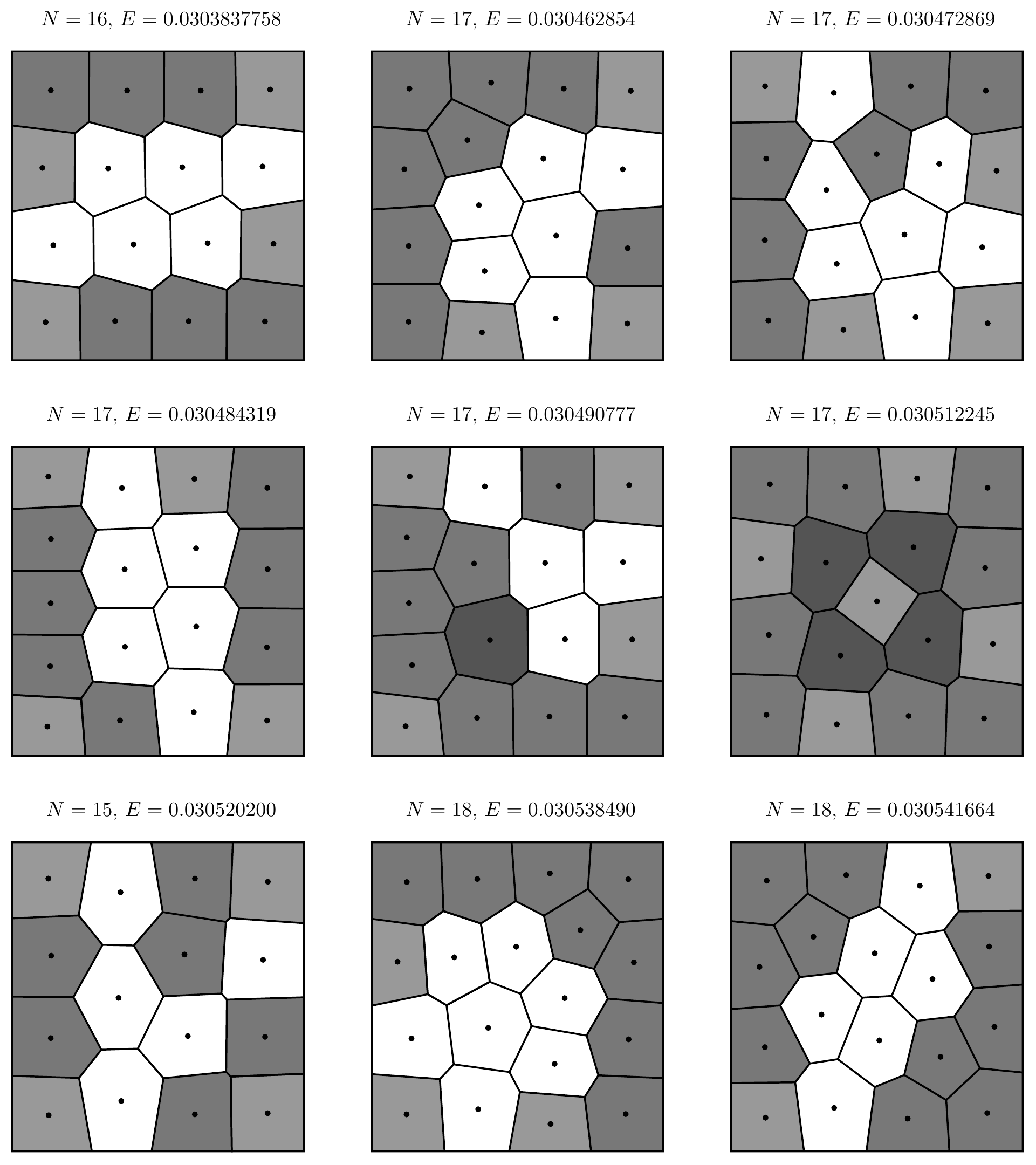}
\caption{\label{fig:flatness} Flatness of energy landscape: Some local minimizers of the energy \eqref{eqn:bc} for $\lambda=0.005$. The polygons are the power cells $P_i$ and the points are the generators $\bx_i$. The weights $w_i$ are not shown. The shading corresponds to the number of sides of the cells.}
\end{figure}
We believe that the top-left figure is a global minimizer. These were generated using $25,000$ random initial conditions to probe the non-convex energy landscape. The energy has infinitely many critical points, e.g., every centroidal Voronoi tessellation of $[0,1]^2$ with cells of equal area (such as the checkerboard configuration) is a critical point. The flatness of the energy landscape can be seen from the energy values in Figure \ref{fig:flatness}.
\begin{figure}[h!]
\includegraphics[width=0.5\textwidth]{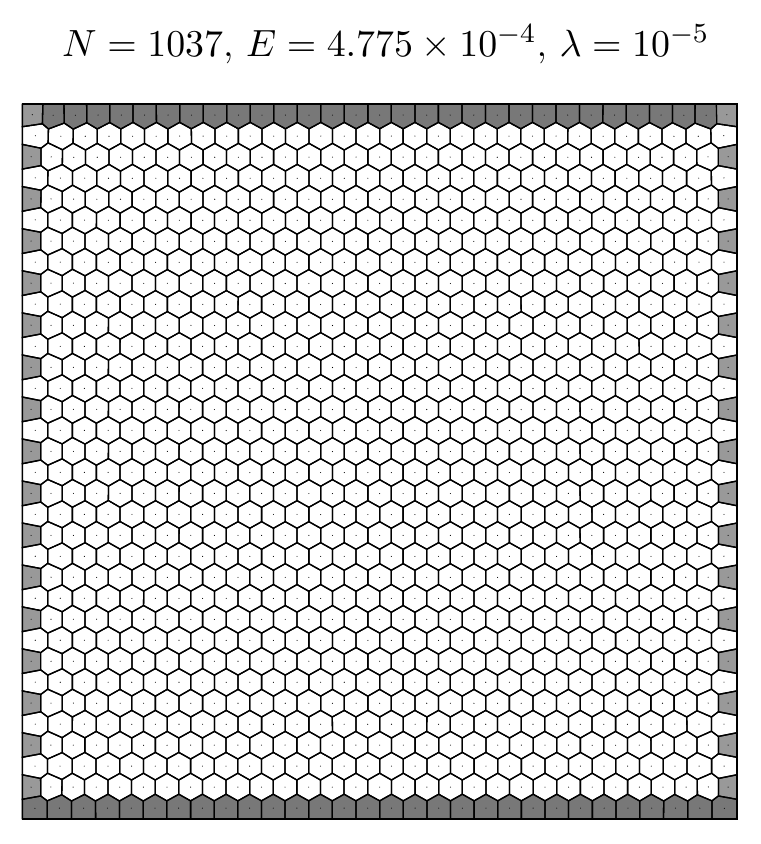}
\includegraphics[width=0.5\textwidth]{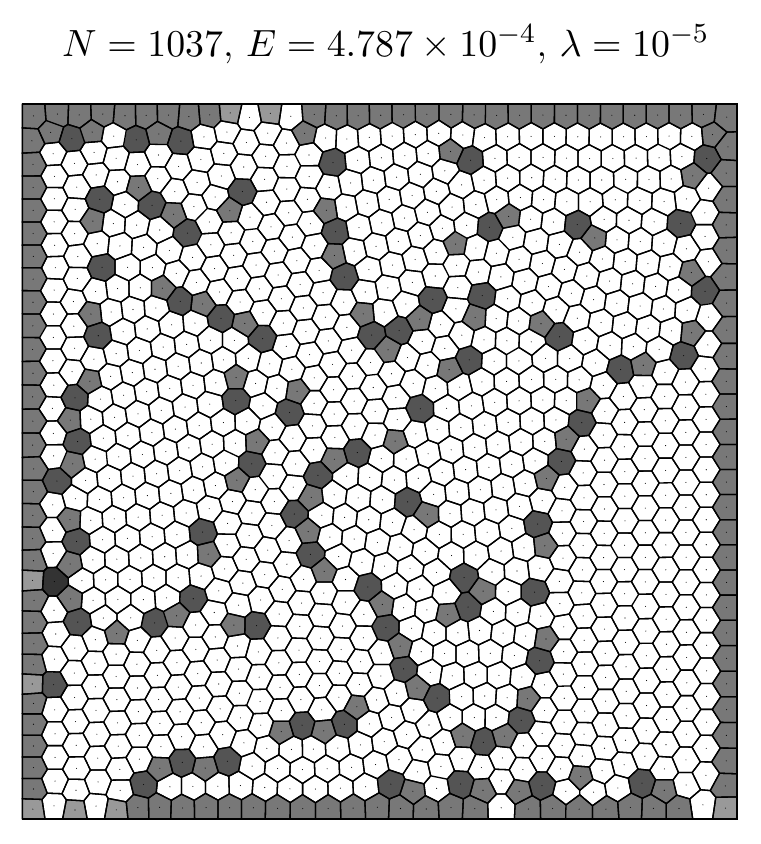}
\caption{\label{fig:flatness_large} Two local minimizers of the energy \eqref{eqn:bc} for $\lambda=10^{-5}$ with $N=1037$ in both cases. In the first case the cell generators were initially arranged in a triangular lattice, in the second case they were distributed randomly.}
\end{figure}

As $\lambda$ decreases it becomes harder to find global minimizers. Figure \ref{fig:flatness_large} shows two local minimizers for $\lambda=10^{-5}$. The figure on the left was obtained by using the triangular lattice as an initial condition. It was proved in \cite{BournePeletierTheil} that the triangular lattice is optimal in the limit $\lambda \to 0$. The figure on the right was obtained with a random initial condition. The `grains' of hexagonal tiling resemble grains in metals. This suggests that energies of the form \eqref{eqn:energy} could be used to simulate material microstructure, for example to produce Representative Volume Elements for finite element simulations \cite{Harrison}.

\subsection{Convergence rate}
In this section we study the rate of convergence of the algorithm to critical points of the energy \eqref{eqn:bc} with $\lambda=0.005$.
Figure \ref{fig:convergence}
 shows the logarithm of the approximate error of the energy plotted against the number of iterations $n$ for three simulations with random initial conditions. The initial number of generators was $N=6, 10, 25$ and there was no elimination of generators throughout the simulations. The approximate error was computed using the value of the energy at the final iteration. The graph shows that the energy converges linearly, meaning that the error at the $n$--th iteration $\varepsilon_n$ satisfies $\varepsilon_{n+1}/\varepsilon_n \to r$, where $r \in (0,1)$ is the rate of convergence. We observe that the rate of convergence decreases as the number of generators increases and that $r \sim 1- \frac{C}{N}$ for some constant $C$.
 In \cite{Du2006} it was found that for the classical Lloyd algorithm with $\rho=1$ in one dimension the rate of convergence of the generators (rather than the energy) is approximately $1-1/(4\pi^2N^2)$. This was found from the spectrum of the derivative of the Lloyd map. In principle the rate of convergence of the generalized Lloyd algorithm could be found using the derivatives given in Proposition \ref{prop:DLloyd}.
We believe that region ($\star$) in the figure is the result of the Lloyd iterates passing close to a saddle point of the energy on the way to a local minimum.
\begin{figure}[h!]
\includegraphics{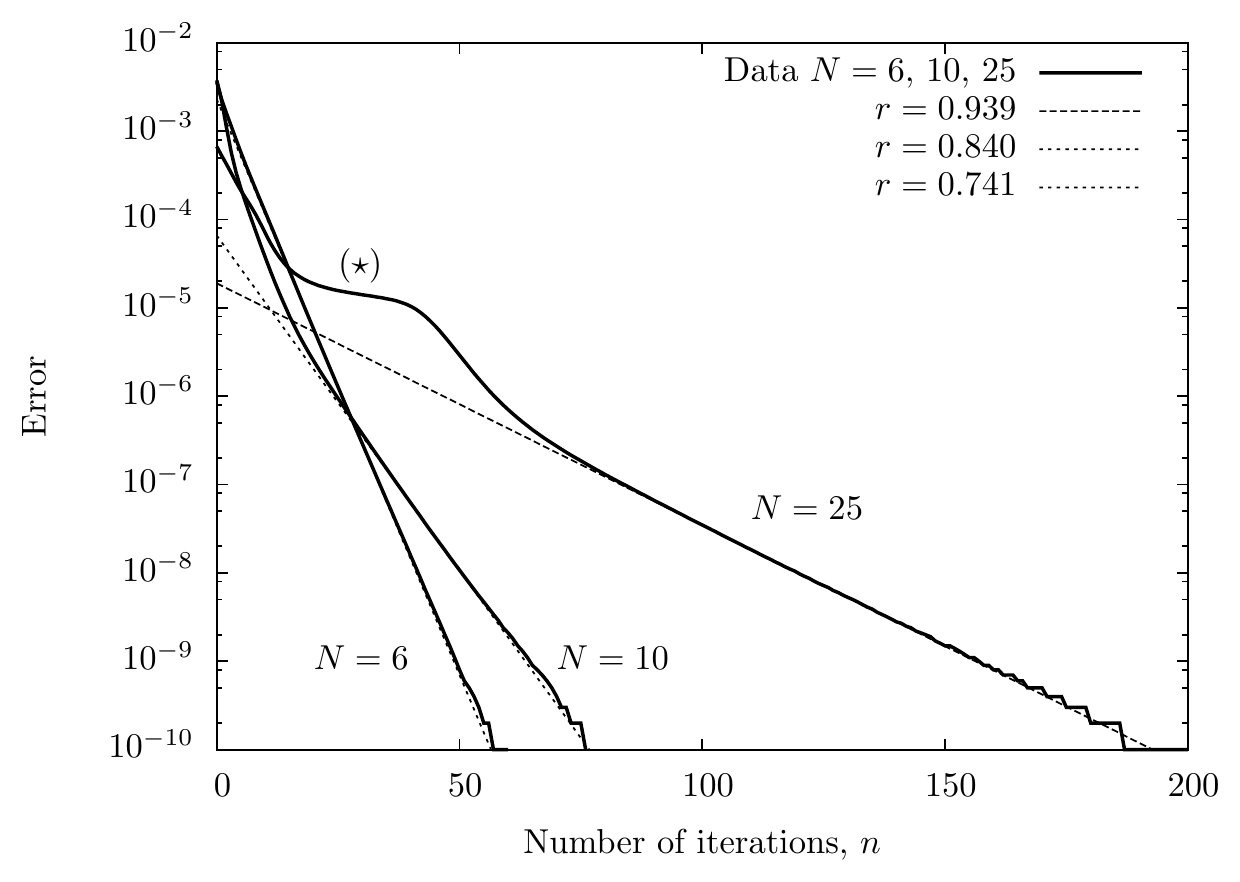}
\caption{\label{fig:convergence} Rate of convergence of the generalized Lloyd algorithm to critical points of the energy \eqref{eqn:bc} with $\lambda=0.005$: Approximate error of the energy against the number of iterations on semi-log axes for three simulations with random initial conditions. The initial number of generators was $N=6,10,25$ and there was no elimination of generators throughout the simulations. We see that the algorithm converges linearly. The rate $r$ was computed by fitting straight lines to the data.}
\end{figure}

\subsection{An optimal location problem with non-constant $\rho$}
\label{Subsec: Non-const}

In the block copolymer example in the previous sections we had $\rho=1$. In an optimal location problem $\rho$ need not be uniform and might represent population density. The term $f(m)$ represents the cost of building or running a facility to serve $m$ individuals. The function $f$ is concave, which represents an economy of scale.

A particular case of interest would be to determine where to locate government agencies (stations) to which people must attend at some rate (for example, a trip to the passport office). Somewhat artificially we may propose that the cost per person of a trip of length $l$ is $\tilde{c} \,c(l/L)$ where $L$ is a representative distance, $\tilde{c}$ is a constant with units of cost per person, and $c$ is a non-dimensional cost function. Let $\{P_i\}_{i=1}^N$ be a power diagram with generators $\{ \bx_i,w_i\}$. We assume that if a person belongs to power cell $P_i$, then they must use the station located at $\bx_i$, and that they visit the station $\omega$ times per year. Then the cost per year $C$ to the people travelling to the locations $\{\bx_i\}_{i=1}^N$ is
$$C= \tilde{c} \omega \sum_{i=1}^N  \int_{P_i} c\left(\frac{|\bx-\bx_i|}{L}\right)\rho\left(\bx\right)\,d\bx.$$
Suppose that the cost per year of running a station that serves $m$ individuals is $s f(m/M)$ where $M$ is a characteristic number of people, $s$ has units of cost per year, and $f$ is a non-dimensional cost function. Using $c(x)=x^2$ we obtain
$$\sum_{i=1}^N \left\{ sf\left(\frac{m_i}{M}\right)+\frac{\tilde{c}\omega}{L^2}\int_{P_i}|\bx-\bx_i|^2\rho\left(\bx\right)\,d\bx \right\}$$
which represents the combined cost per year of running the stations and the travel costs of the users. This cost must be minimised. By rescaling we obtain the energy
\[
E(\{ \bx_i, w_i \}) = \sum_{i=1}^N \left\{ \lambda f(m) + \int_{P_i} |\bx-\bx_i|^2\rho\left(\bx\right)\,d\bx \right\}.
\]
The parameter $\lambda$ is a measure of the cost of running a station compared to the cost incurred by the individuals using the station; small values of $\lambda$ represent a station that is low in cost to run and large values of $\lambda$ represent a station that is high in cost to run (perhaps because of infrequent visits by its users). As a concrete example we take European population data covering metropolitan France and ask where to cite stations for different choices of $\lambda$, using the function $f(m)=- m \log m$. Figure \ref{fig:france_example} shows the results for two different choices of $\lambda$, loosely corresponding to departments/regions and their centres of administration and extents.
\begin{figure}[h!]
\begin{center}
\setlength\fboxsep{0pt}
\setlength\fboxrule{0.5pt}
\fbox{\includegraphics[width=0.75\textwidth]{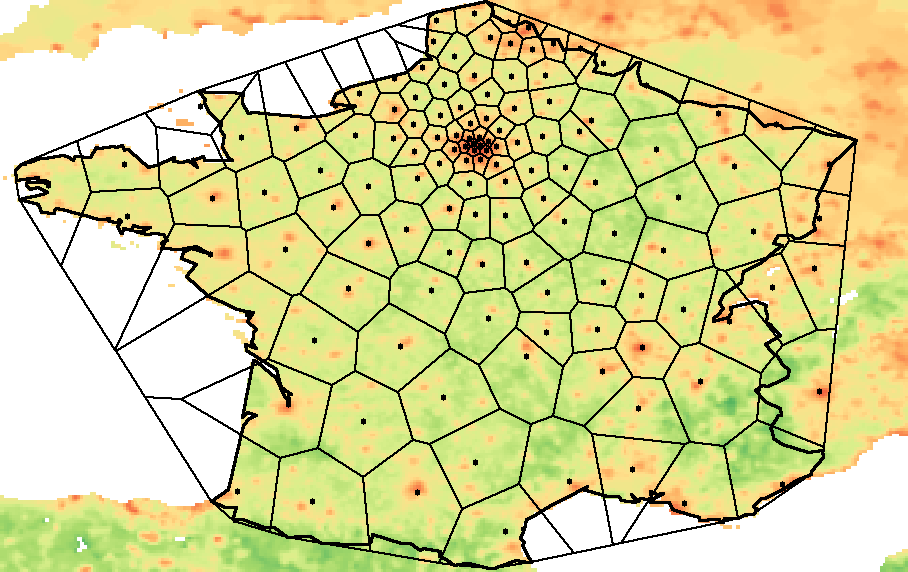}}
\fbox{\includegraphics[width=0.75\textwidth]{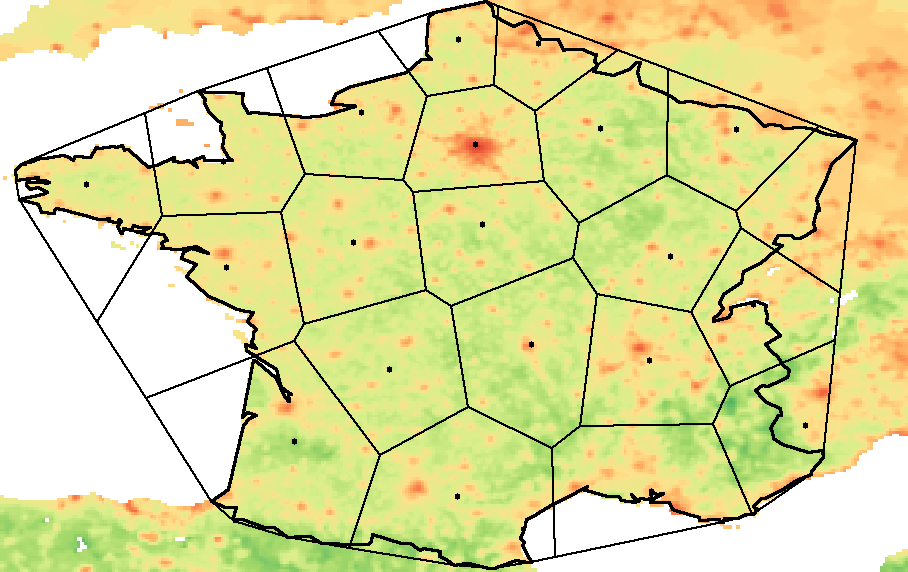}}
\caption{\label{fig:france_example} An example illustrating the algorithm applied to the convex hull of metropolitan France (recall that we require the region $\Omega$ to be convex). The colour represents the population density $\rho$, red is high density and green is low density (population data obtained from Eurostat). Areas in which no data was available have been assigned a population density of zero, for example in the seas and oceans. The figures were produced using Algorithm \ref{algo:genLloyd} with an initial condition of 5000 random generators and 9000 iterations. The top figure has $\lambda=0.01$ and the bottom figure has $\lambda=20$. In the particular instances here, the final local minimum of the energy has 128 cells when $\lambda=0.1$ and 21 cells when $\lambda=20$.}
\end{center}
\end{figure}

\subsection{An example in three dimensions: crystallization}
\label{Subsec:3D}
In this section we implement the generalized Lloyd algorithm in three dimensions for the block copolymer model from \S\ref{Subsubsec: block copolymer}:
\begin{equation}
\label{eqn:bc3D}
E \left( \{\bx_i,w_i\} \right) = \sum_{i=1}^N \left\{ \lambda  m_i^{\frac{2}{3}} + \int_{P_i} |\bx-\bx_i|^2 \,d \bx \right\}
\end{equation}
where $\bx_i \in \Omega \subset \mathbb{R}^3$. It was conjectured in \cite{BournePeletierRoper} that global minimizers of $E$ tend to a body-centred cubic (BCC) lattice as $\lambda \to 0$, meaning that the set $\{ \bx_i \}$ tends to a BCC lattice and $w_i \to 0$ for all $i$. This conjecture was motivated by block copolymer experiments and by results for the special case $\lambda=0$, $N \to \infty$: in \cite{Barnes1983} it was proved that the BCC lattice has asymptotically the lowest energy amongst all lattices and \cite{Du2005} provided numerical evidence that it has asymptotically the lowest energy amongst all possible configurations $\{ \bx_i \}$. Simulations in \cite[Sec.~4.1]{BournePeletierRoper} in two dimensions demonstrate that global minimizers of $E$ for $\lambda>0$ (centroidal power diagrams) are close to global minimizers of $E$ for $\lambda=0$ (centroidal Voronoi tessellations). In this section we give further numerical support for the conjecture.

It is computationally expensive to study the limit $\lambda \to 0$ in three dimensions since the optimal number of generators grows like $N \sim \lambda^{-1}$.
 Instead we restrict our attention to the case where $\Omega$ is a periodic cube. Figure \ref{fig:bcccell} shows a representative Voronoi cell generated by the BCC lattice, which is a truncated octahedron (Kelvin proposed a deformed version of the truncated octahedron as a candidate for three-dimensional foams).
 If $N$ is chosen appropriately, then $N$ of these cells fit exactly into the periodic cube and there is no boundary layer. If $\{\bz_i\}_{i=1}^N$ are the centres of these cells, then $\{ \bz_i, 0 \}_{i=1}^N$ is a critical point of $E$ (because it is a centroidal Voronoi tessellation and all cells have the same mass). We study its stability using the generalized Lloyd algorithm.


\begin{figure}[h!]
\begin{center}
\includegraphics[width=0.5\textwidth]{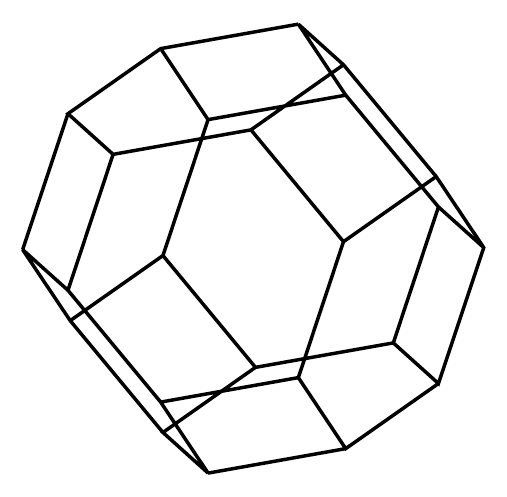}
\end{center}
\caption{\label{fig:bcccell}A Voronoi cell generated by the BCC lattice.}
\end{figure}

\begin{figure}[h!]
\begin{center}
\includegraphics[width=0.4\textwidth]{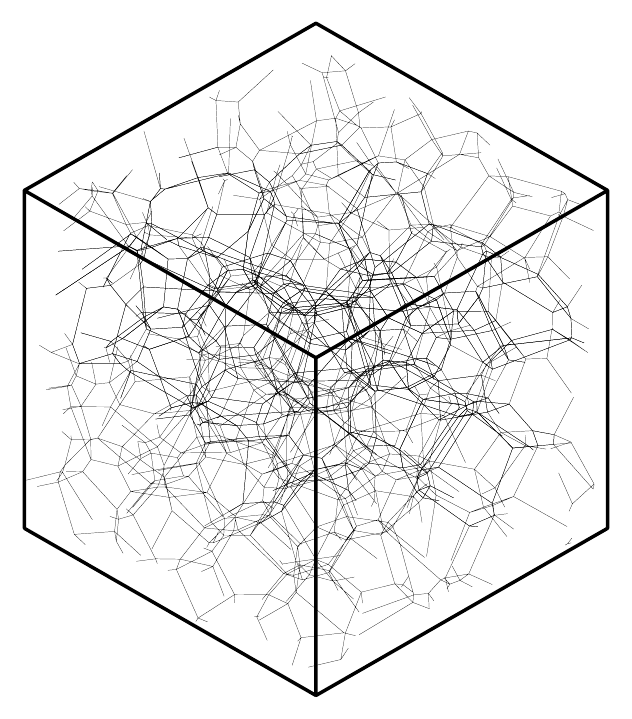}
\includegraphics[width=0.4\textwidth]{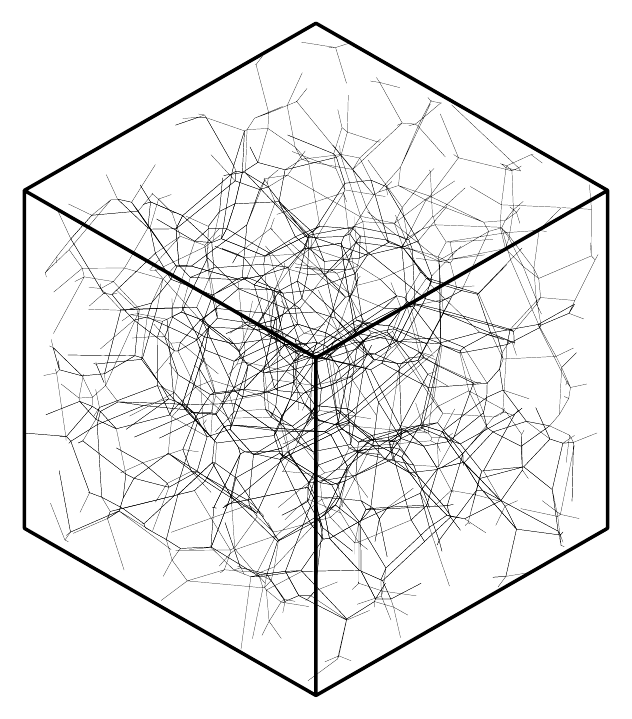}\\
\includegraphics[width=0.4\textwidth]{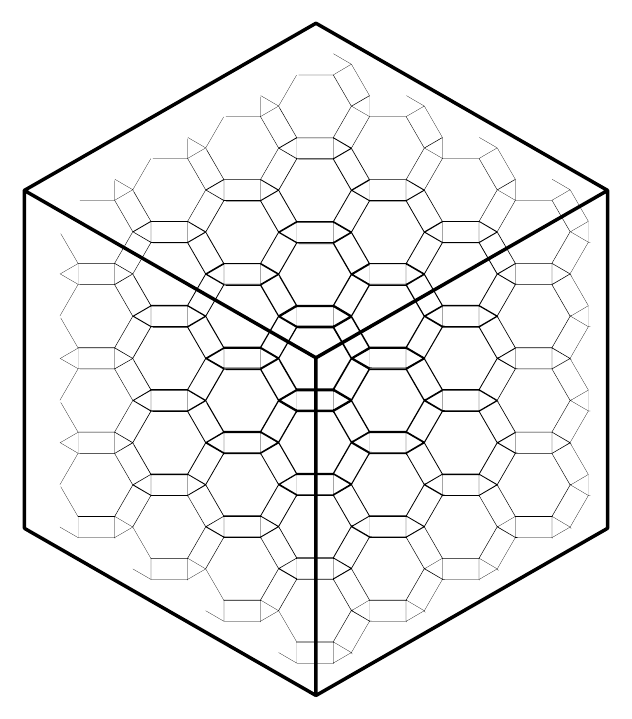}
\includegraphics[width=0.4\textwidth]{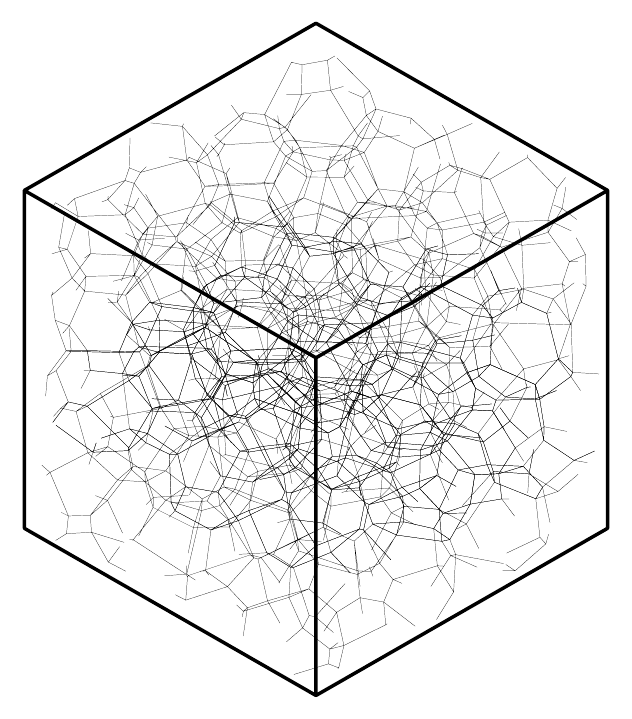}
\end{center}
\caption{\label{fig:bcc_perturb}
Numerical evidence that the BCC lattice is a local minimizer of \eqref{eqn:bc3D} for $\lambda=10^{-3}$.
Left column: the initial condition (top-left) is a perturbation of the BCC lattice (a perturbation of both the generator locations and weights). The perturbation is small enough that the algorithm converges to the BCC lattice (bottom-left, configuration after $2000$ iterations). Right column: the initial condition (top-right) is a large enough perturbation of the BCC lattice to cause the algorithm to converge to a different local minimum (bottom-right, configuration after $2000$ iterations). In all figures $N=128$.}
\end{figure}

We implemented the algorithm in C++ using the Voro++ software library to compute power diagrams \cite{Rycroft2009}. We found that the BCC lattice is stable under small perturbations; if the initial condition $ (\bX^0, \bw^0)$ is taken to be a small enough perturbation of the BCC lattice, then the generalized Lloyd algorithm converges back to the BCC lattice. See Figure \ref{fig:bcc_perturb}, left column (the initial configuration is top-left, the final configuration is bottom-left). This suggests that the BCC lattice is at least a local minimizer of the energy. Under larger perturbations the Lloyd algorithm converges to a different critical point with a higher energy. See Figure \ref{fig:bcc_perturb}, right column (the initial configuration is top-right, the final configuration is bottom-right). We also tested the energy of the BCC lattice against the energy of several common lattices and found that it was lower in each case. Due to the non-convexity and flatness of the energy landscape, however, the conjecture requires a more detailed numerical study.



\appendix

\section{Useful implementation formulas for the case $\rho=$ constant}
\label{append:poly}
In two dimensions in the special case that $\rho\left(\bx\right)=\rho_0$, a constant, the integrals defining the mass and centroid of a polygonal cell can be computed without using a quadrature rule; they can be expressed explicitly as functions of the vertices of the cell.


Consider a polygon $P$, lying in the plane with normal $\bm{k}$ with vertices $\bm{v}_k$ for $k=0,\ldots,N-1$. We use the notation $k\oplus m=(k+m)\mod N$. The vertices are numbered anti-clockwise around the boundary (with respect to the normal $\bm{k}$). We denote the outward normal to the polygon
by $\bnu$. On edge $k$ (with end points $\bm{v}_k$ and $\bm{v}_{k\oplus1}$ and length $L_k$) $\bnu$ can be written in terms of the vertices as
\begin{equation}
\nonumber
\bnu=\frac{1}{L_k}\left(\bm{v}_{k\oplus1}-\bm{v}_k\right)\times\bm{k}.
\end{equation}
As the density is constant, we can use the Divergence Theorem to express integrals over polygons as edge integrals:
\begin{equation}
\nonumber
\int_P\rho_0\,d\bx=\frac{1}{2}\rho_0\int_{\partial P} \bm{x}\cdot\bnu\,ds=\frac{1}{2}\rho_0\sum_{\text{edges $k$}}\int_{\bm{v}_k}^{\bm{v}_{k\oplus 1}}\bm{x}\cdot\bnu\,ds.
\end{equation}
Along an edge, using an arc-length parametrization,
\begin{equation}
\nonumber
\bm{x}=\frac{s}{L_k}\bm{v}_{k\oplus 1}+\frac{(L_k-s)}{L_k}\bm{v}_k, \quad
\bm{x}\cdot\bnu=\frac{1}{L_k}\left[\bm{k}\cdot\left(\bm{v}_k\times\bm{v}_{k\oplus 1}\right)\right].
\end{equation}
Since $\bm{x}\cdot\bnu$ is independent of $s$ we obtain
\begin{equation}
\int_P\rho_0\,d\bx=\frac{1}{2}\rho_0\sum_{\text{edges $k$}}\left[\bm{k}\cdot\left(\bm{v}_k\times\bm{v}_{k\oplus 1}\right)\right].
\end{equation}
Similarly
\begin{equation}
\int_P\rho_0\bm{x}\,d\bx=\frac{1}{6}\rho_0\sum_{\text{edges $k$}}\left(\bm{v}_k+\bm{v}_{k\oplus 1}\right)\left[\bm{k}\cdot\left(\bm{v}_k\times\bm{v}_{k\oplus 1}\right)\right]
\end{equation}
and
\begin{align}
&\int_P\rho_0\bx\otimes\bx\,d\bx=\\
\nonumber
&\frac{1}{24}\rho_0\sum_{\text{edges $k$}}\left[2\bm{v}_k\otimes\bm{v}_k+2\bm{v}_{k\oplus 1}\otimes\bm{v}_{k\oplus 1}+\bm{v}_k\otimes\bm{v}_{k\oplus 1}+\bm{v}_{k\oplus 1}\otimes\bm{v}_k\right]\left[\bm{k}\cdot\left(\bm{v}_k\times\bm{v}_{k\oplus 1}\right)\right].
\end{align}

%


\medskip

\paragraph{Acknowledgements} The generalized Lloyd algorithm was derived for a special case in collaboration with Mark Peletier \cite{BournePeletierRoper}. Voro++ \cite{Rycroft2009} was used to generate the power diagrams for the three-dimensional simulations in Section \ref{Subsec:3D}. All plots were prepared using Gnuplot.

\medskip

\bibliographystyle{siam}
\bibliography{dbsrpaper}

\newpage

\end{document}